\documentclass[review,12pt]{article}
\usepackage[top=1in, bottom=1in, left=1in, right=1in]{geometry}
\usepackage{lipsum}
\usepackage{amsfonts}
\usepackage{graphicx}
\usepackage{tabularx}
\usepackage{epstopdf}
\usepackage{algorithmic}
\usepackage{xcolor}
\usepackage{mathtools}
\usepackage{amssymb,amsthm}

\usepackage{subcaption}     % Required for creating subfigures
\usepackage[numbers]{natbib}
\usepackage[colorlinks,
            linkcolor=blue,
            anchorcolor=blue,
            citecolor=blue
            ]{hyperref}

\usepackage{lineno}    % 核心行号宏包
\usepackage{blindtext} % 仅用于生成示例文本，实际使用时可删除

\ifpdf
  \DeclareGraphicsExtensions{.eps,.pdf,.png,.jpg}
\else
  \DeclareGraphicsExtensions{.eps}
\fi

\numberwithin{equation}{section}
\theoremstyle{plain} 
\newtheorem{theorem}{Theorem}[section]
\newtheorem{corollary}[theorem]{Corollary}
\newtheorem{lemma}[theorem]{Lemma}
\newtheorem{assumption}[theorem]{Assumption}

\theoremstyle{remark}
\newtheorem{remark}[theorem]{Remark}

\newcommand{\ii}{\mathrm{i}}

\newcommand{\real}{\mathbb{R}}

\newcommand{\natu}{\mathbb{N}}

\newcommand{\bsb}{\boldsymbol{b}}

\newcommand{\bsalpha}{\boldsymbol{\alpha}}

\newcommand{\bsk}{\boldsymbol{k}}

\newcommand{\bsu}{\boldsymbol{u}}

\newcommand{\bsx}{\boldsymbol{x}}
\newcommand{\bsy}{\boldsymbol{y}}

\newcommand{\bszero}{\boldsymbol{0}}
\newcommand{\bsone}{\boldsymbol{1}}
\newcommand{\bsell}{{\boldsymbol{\ell}}}

\newcommand{\bsgamma}{\boldsymbol{\gamma}}

\newcommand{\rd}{\,\mathrm{d}}
\newcommand{\mrd}{\mathrm{d}}

\newcommand{\e}{\mathbb{E}}

\newcommand{\walk}{\mathrm{wal}_k}
\newcommand{\walkj}{\mathrm{wal}_{k_j}}
\newcommand{\walbsk}{\mathrm{wal}_{\bsk}}

\newcommand{\inner}[2]{\langle #1, #2 \rangle_{H^{-1}(D),H_0^1(D)}}

\newcommand{\cp}{\mathcal{P}}

\newcommand{\newnorm}[1]{\Vert #1\Vert}
\newcommand{\supp}{\boldsymbol{s}}

\newcommand{\mix}{\mathrm{mix}}
\newcommand{\itr}{\mathbb{I}}

\usepackage{amsopn}

\begin{document}

\title{Uncertainty quantification using importance-sampled quasi-Monte Carlo with dimension-independent convergence rates
\thanks{The work of the third author was funded by the Guangdong Basic and Applied Basic Research Foundation grant  2024A1515011876 and 2025A1515011888.}}

% Authors: full names plus addresses.
\author{Zexin Pan\thanks{Institute of Fundamental and Transdisciplinary Research, Zhejiang University, Zhejiang 310058, People's Repulic of China
  (zep002@zju.edu.cn).} \and Du Ouyang\thanks{Department of Mathematical Sciences, Tsinghua University, Beijing 100084, People's Republic of China (oyd21@mails.tsinghua.edu.cn).} \and Zhijian He\thanks{Corresponding author. School of Mathematics, South China University of Technology, Guangzhou 510641, People's Repulic of China (hezhijian@scut.edu.cn).}
  }
\maketitle
% REQUIRED
\begin{abstract}
Quasi-Monte Carlo (QMC) integration over unbounded domains $\mathbb{R}^s$ remains challenging due to the high dimensionality of sampling space and the boundary growth of the integrand. In applications such as uncertainty quantification (UQ), the dimension $s$ can reach hundreds or even thousands. To restore the efficiency of quadrature rules in high dimensions, constructive QMC methods like lattice rules have been successfully developed within the framework of weighted function spaces. In contrast to designing problem-specific quadrature points, this paper proposes transforming the underlying integrand to accommodate the off-the-shelf scrambled nets (a construction-free randomized QMC method) via the
boundary-damping importance sampling (BDIS) proposed by \cite{pan:2025}.
We provide a rigorous analysis of the dimension-independent convergence rate of BDIS-based scrambled nets while covering a broader class of unbounded functions than that in \cite{pan:2025}. By exploiting the dimension structure of the parametric input random field, the proposed $n$-point quadrature rule achieves a dimension-independent mean squared error rate of $O(n^{-1-\alpha^*+\varepsilon})$ on standard UQ problems in elliptic partial differential equations (PDEs), where $\varepsilon>0$ is arbitrarily small and $\alpha^*\in (0,1)$ reflects the regularity with respect to the parametric variables. Numerical experiments on elliptic PDEs with high-dimensional parameters further demonstrate the effectiveness of the method.

\smallskip

\noindent \textbf{Key words:}
Quasi-Monte Carlo, importance sampling, uncertainty quantification, parametric PDE

\noindent \textbf{MSC codes:}
41A63, 65D30, 97N40
\end{abstract}

\section{Introduction}\label{sec:intro}
In this paper, we study quasi-Monte Carlo (QMC) quadrature rules for computing integrals over unbounded domain $\real^s$ of the form
\begin{equation}
\mu=\int_{\real^s} f(\bsx) \prod_{j=1}^s \varphi(x_j)\rd \bsx,\label{eq:mu}
\end{equation}
where $f:\real^s \to \real$ is a real-valued integrand and $\varphi:\real\to \real$ is a probability density function.
Such a fundamental problem roots in many fields,  such as  financial engineering \cite{glas:2004,lecuyer:2009,zhang:2021} and uncertainty quantification (UQ) \cite{graham:2015,kaarnioja:2022,kuo:nuye:2016,schillings:2020}. Although QMC has the potential to accelerate the convergence rate of Monte Carlo (MC), 
its application to high-dimensional functions is still a challenging task because of the often occurring curse of dimensionality. In last decades, researchers made great efforts to understand how high dimensionality could be tackled successfully by QMC methods. On the one hand, the concept of ``low effective dimension'' plays an important role to explain the success of QMC in applications with high dimensions, meaning that high-dimensional functions may in fact depend mainly on a few leading variables or the interaction of a small number of variables \cite{cafmowen,effdimsobononper,wang:fang:2003}. On the other hand, weighted function spaces  were introduced by \cite{sloan:1998} for handling high-dimensional problems, in which a set of parameters (namely ``weights") are built into the function space norm to model the relative importance between different subsets of variables.
Under certain conditions on the weights, a class of constructive QMC methods known as lattice rules  with component-by-component (CBC) construction achieves a dimension-independent error bound \cite{nichols_2014_fast} and has proven successful in many UQ applications, including elliptic  partial differential equations (PDEs) with random diffusion coefficients \cite{graham:2015,guth:2025, kuo:nuye:2016} and elliptic inverse problems \cite{herrmann:2021,scheichl:2017}. 
The philosophy behind the applications is to choose the weights according to the dimension structure of the underlying integrands and then construct lattice rules via CBC \cite{sloan:2002} or fast CBC  \cite{nuyens:2006} that are amenable to high-dimensional problems. It is crucial to learn the dimension structure of the problems at hand. Rather than constructing quadrature points via CBC, our strategy is to transform the underlying integrands via dimension information-based importance sampling (IS), which can then be easily handled by off-the-shelf QMC methods such as scrambled nets \cite{rtms,owen96}. This paper builds upon the boundary-damping IS recently proposed by the present authors \cite{pan:2025}. We first establish an improved bound on the dimension-independent convergence rate of scrambled net integration over $\real^s$ under the boundary-damping IS, and then generalize the results to elliptic
PDEs with general random diffusion coefficients and weighted function spaces.

Throughout this paper, we use $\bsx$ for coordinates in $\real^s$ and $\bsu$ for coordinates in $\itr^s$, where $\itr=(0,1)$.  
To estimate $\mu$ in \eqref{eq:mu}, we consider IS of the form
\begin{equation}\label{eqn:muhat}
 \hat{\mu}=\frac{1}{n}\sum_{i=0}^{n-1} w(\bsu_i)  f\circ T(\bsu_i),   
\end{equation}
where $T(\bsu):\itr^s\to \real^s$ is called the transport map, $\{\bsu_0,\dots,\bsu_{n-1}\} \subseteq \itr^s$ is a set of quadrature points, and $w(\cdot)$ is known as likelihood ratio or  \textit{weight function}. The transport map allows us to generate the proposal via standard uniform distribution tailored to QMC quadrature rules. In this paper, we focus on using scrambled digital nets as the quadrature points following the setting of randomized QMC (RQMC) for ease of error estimation \cite{vLEC02a,owen:2013}.

Pan et al.  \cite{pan:2025} considered transport maps with independent components, i.e., $T(\bsu)=(T_1(u_1),\dots, T_s(u_s))$. If the transport map $T_j:\itr\to \real$ is differentiable, the weight function then has the product form
$$w(\bsu)=\prod_{j=1}^s w_j(u_j)  \quad \text{for} \quad w_j(u)=T_j'(u)  \varphi\circ T_j(u) .$$
We thus have 
\begin{equation}\label{eqn:Tj}
    T_j(u)=\Phi^{-1}\Big(\int_{0}^u w_j(t)\rd t\Big),
\end{equation}
where $\Phi(x)=\int_{-\infty}^x \varphi (y) \rd y$ be the cumulative distribution function (CDF) and $\Phi^{-1}:\itr \to \real$ be the inverse CDF (quantile function). Rather than choosing IS proposals from a family of distributions as in the setting of MC \cite{glas:2004,owen:2013},  Pan et al. \cite{pan:2025} proposed to choose weights $w_j(\cdot)$ from a set of parametrized functions, i.e., $w_j(u)=w_{\theta_j}(u)$ with  $\theta_j\in (0,1/2]$. Given the parameter $\theta_j$, the transport map $T_j$ is determined by \eqref{eqn:Tj}.
For $\theta\in(0,1/2]$, Pan et al. \cite{pan:2025} particularly took
\begin{equation}\label{eqn:wtheta}
w_\theta(u)=\begin{cases}
    (1-\theta)^{-1}\eta(u/\theta), &\text{ if } u\in (0,\theta/2] \\
    (1-\theta)^{-1}(1-\eta(1-u/\theta)) , &\text{ if } u\in (\theta/2,\theta) \\
    (1-\theta)^{-1}, &\text{ if } u\in [\theta,1/2] \\
    w_\theta(1-u), &\text{ if } u\in (1/2,1), \\
\end{cases}  
\end{equation}
and
\begin{equation}\label{eq:tildeetap}
\eta(u)=\eta_p(u)=\begin{cases}
    2^{-p-2} u^{-p-1}\exp(2^p-u^{-p}), &\text{ if } u\in (0,1/2]\\
    0, &\text{ if } u=0
\end{cases} \text{ for } \ p\geq 1.
\end{equation}
We in addition let $T_j(u)=\Phi^{-1}(u)$ when $\theta_j=0$, which is the usual transport map used by inversion methods \cite{cui2025secondorderinterlacedpolynomial,dick2025quasi, nichols:2014, haltavoid, ye2025medianqmcmethodunbounded}.  As shown in \cite{pan:2025}, estimating $\mu$ is equivalent to integrating 
$$f^w(\bsu):=w(\bsu) f\circ T(\bsu)$$
over $\itr^s$. The weight function $w_\theta(u)$ damps the boundary growth of the integrand $f^w(\bsu)$, so that a faster convergence rate of RQMC can be achieved when the original integrand $f(\bsx)$ grows unboundedly as $\bsx$ goes to infinity, hence the name boundary-damping IS (BDIS).
The mean squared error (MSE) convergence rate of the BDIS method is also established when $f\in W^{1,q}_{\mix}(\real^s,\varphi)$ with $q=2$ or $f\in W^{1,q,\infty}_{\mix}(\real^s,\varphi)$ with $q\in (1,2]$, where a smaller $q$ indicates more severe boundary growth (see the definitions of function norms in Section~\ref{sec:main}). A limitation of the theory developed in \cite{pan:2025} is that the convergence rate does not improve for integrands in $W^{1,q}_{\mix}(\real^s,\varphi)$ or $W^{1,q,\infty}_{\mix}(\real^s,\varphi)$ with $q>2$, which limits its applicability in UQ problems. It is therefore of interest to investigate the performance of BDIS on integrands with mild boundary growth (i.e., large $q$) or even no significant boundary growth (i.e., arbitrarily large $q$).

The contribution of this paper is threefold. Firstly, we provide an improved upper bound on the dimension-independent convergence of the BDIS method for $f\in W^{1,q}_{\mix}(\real^s,\varphi)$ with any $q>1$, making our theory suitable for functions with mild boundary growth.  Secondly, we extend our analysis to PDEs with a general parametric input random field. Lattice rule-based QMC methods have proved effective in the context of elliptic PDEs with lognormal random diffusion coefficients \cite{graham:2015,herrmann:2019}.
More recently, Guth and Kaarnioja \cite{guth:2025} studied elliptic PDEs with Gevrey regular inputs \cite{chernov:2024}, where the parametric inputs follow a generalized $\beta$-Gaussian distribution, and analyzed the QMC integration error for assessing PDE response statistics using randomly shifted rank-1 lattice rules. Our analysis also covers such PDE problems, and shows the BDIS-based scrambled nets with $n$ points achieve a dimension-independent MSE rate of $O(n^{-1-\alpha^*+\varepsilon})$, where $\varepsilon>0$ is arbitrarily small and $\alpha^*\in (0,1)$ depends on the regularity with respect to the parametric variables. Thirdly, we generalize the convergence results to the setting of  weighted function spaces, thereby bridging our method with constructive QMC methods such as lattice rules \cite{graham:2015,guth:2025,kuo:nuye:2016}.

The rest of this paper is organized as follows. Section~\ref{sec:main} provides the main convergence results, improving those in \cite{pan:2025}. Section~\ref{sec:proofs} gives the proof of the main theorem in Section~\ref{sec:main}. Section~\ref{sec:pdes} shows how to apply our method to  parameterized elliptic PDEs. Section~\ref{sec:weightedSobolev} generalizes the convergence results using weighted Sobolev norms. Numerical results are presented in Section~\ref{sec:numer} to illustrate the effectiveness of our proposed method.  Section~\ref{sec:conclusion} concludes this paper.

\section{Main results}\label{sec:main}
We follow the notations in \cite{pan:2025}.  For a vector $\bsx\in\real^s$ and a subset $v\subseteq 1{:}s$, $\bsx_v$  denotes the subvector of $\bsx$ indexed by $v$, while $\bsx_{v^c}$ denotes the subvector indexed by $1{:}s\setminus v$. Let $\natu$ be the set of positive integers and $\natu_0 = \natu\cup\{0\}$. Let $\Vert\bsx\Vert_q=(\sum_{j=1}^s |x_j|^q)^{1/q}$ for $q>0$. For a nonempty set $v\subseteq 1{:}s$, define the mixed derivative by
$$\partial^v f(\bsx) = \left(\prod_{j\in v}\frac{\partial}{\partial x_j}\right)f(\bsx),$$
and define $\partial^\emptyset f(\bsx)=f(\bsx)$ by convention. All constants in this paper have an implicit dependency on $\varphi$ and we suppress it from the notation for simplicity.

This paper will use some function norms as defined in the following:
\begin{align*}
    \Vert f \Vert_{L^q(\itr^s)}&=\begin{cases}
        \Big(\int_{\itr^s} |f(\bsu)|^q \rd \bsu \Big)^{1/q}&0<q<\infty\\
        \sup_{u\in \itr^s}|f(\bsu)|&q=\infty
    \end{cases},\\
        \Vert f \Vert_{W^{1,q}_{\mix}(\itr^s)}&=\Big(\sum_{v\subseteq 1{:}s} \Vert \partial^v  f \Vert^q_{L^q(\itr^s)}\Big)^{1/q},\\
    \Vert f \Vert_{L^q(\real^s,\varphi)}&=\Big(\int_{\real^s} |f(\bsx)|^q  \prod_{j=1}^s \varphi(x_j) \rd\bsx\Big)^{1/q},\\
    \Vert f \Vert_{W^{1,q}_{\mix}(\real^s,\varphi)}&=\Big(\sum_{v\subseteq 1{:}s} \Vert \partial^v  f \Vert^q_{L^q(\real^s,\varphi)}\Big)^{1/q}.
\end{align*}
Pan et al.  \cite{pan:2025} additionally considered the first-order mixed $L^\infty$ norm:
\begin{align*}
    \Vert f \Vert_{L^{q,\infty}(\real^s,\varphi)}&=\Big(\sup_{\bsx\in \real^s} |f(\bsx)|^q  \prod_{j=1}^s \frac{\varphi(x_j)}{\varphi_\infty}  \Big)^{1/q},\\
    \Vert f \Vert_{W^{1,q,\infty}_{\mix}(\real^s,\varphi)}&=\Big(\sum_{v\subseteq 1{:}s} \Vert \partial^v  f \Vert^q_{L^{q,\infty}(\real^s,\varphi)}\Big)^{1/q},
\end{align*}
where $\varphi_\infty:=\sup_{x\in \real} \varphi(x)<\infty$. They studied the BDIS method for $f\in W^{1,q}_{\mix}(\real^s,\varphi)$ with $q=2$ or $f\in W^{1,q,\infty}_{\mix}(\real^s,\varphi)$ with $q\in (1,2]$. In the latter case, the function $f$ is more challenging for   QMC integration and may not belong to $L^2$ due to its rapid boundary growth. Lemma 2.2 of \cite{pan:2025} shows that $W^{1,q,\infty}_{\mix}(\real^s,\varphi)\subseteq W^{1,q'}_{\mix}(\real^s,\varphi)$ whenever $q>q'\geq 1$. To generalize these results, we study the case $f\in W^{1,q}_{\mix}(\real^s,\varphi)$ for any $q>1$.
We will need the following version of Sobolev imbedding theorem.

\begin{lemma}\label{lem:Sobolevimbedding}
    For $f\in W^{1,1}_{\mix}(\itr^s)$, $\Vert f \Vert_{L^\infty(\itr^s)}\leq \Vert f \Vert_{W^{1,1}_{\mix}(\itr^s)}.$
\end{lemma}
\begin{proof}
    See Appendix~\ref{appA}.
\end{proof}

To learn the dimension structure of the integrand, we make use of ANOVA decomposition. Following the framework in \cite{kuo2010decompositions}, we introduce the generalized ANOVA decomposition. Let $P_j:L^1(\real^s,\varphi)\to L^1(\real^s,\varphi)$ denote the integration operator
$$P_j(f)(\bsx)=\int_{\real} f(\bsx) \varphi(x_j) \rd x_j \text{ for } \bsx \in \real^{s}.$$
Notice that $P_j(f)$ does not depend on $x_j$ and $P^2_j=P_j$. We further define the iterated integration operator $P_v=\prod_{j\in v}P_j$. By Fubini's theorem, $P_j$ in $P_v$ can be applied in any order. By convention, $P_\emptyset=I$ is the identity operator. The ANOVA decomposition of $f\in L^1(\real^s,\varphi)$  is given by 
\begin{equation}\label{eqn:anovadecomp}
 f(\bsx)=\sum_{v\subseteq 1{:}s} f_v(\bsx),  
\end{equation}
where $$f_v(\bsx)=\Big(\prod_{j\in v} (I-P_j)\Big) P_{1{:}s\setminus v} f(\bsx) = P_{1{:}s\setminus v} f(\bsx)-\sum_{w\subsetneq v} f_w(\bsx).$$
It is clear that $f_v(\bsx)$ depends only on $\bsx_v$ and hence we may write $f_v(\bsx_v)$ instead. 
%It follows that if $j\in v$,
% \begin{equation}\label{eqn:ANOVAPj}
%   P_j (f_v)=(P_j-P_j^2)\Big(\prod_{j'\in v, j'\neq j} (I-P_{j'})\Big) P_{1{:}s\setminus v} f =0.  
% \end{equation}
The next lemma comes from \cite[Lemma 2.4]{pan:2025}.
\begin{lemma}
    If $f\in W^{1,q}_{\mix}(\real^s,\varphi)$, then $f_v\in W^{1,q}_{\mix}(\real^s,\varphi)$ for all $v\subseteq 1{:}s$. 
\end{lemma}

As mentioned in Section~\ref{sec:intro}, the quadrature points in the estimator \eqref{eqn:muhat} we use is scrambled nets. We next introduce scrambled $(t,m,s)$-nets in base $b\ge 2$ briefly; see \cite{dick:pill:2010} for more details.
For $t,\ m\in\natu_0$ and an integer $b\ge 2$, a  set $\cp:=\{\bsu_0,\dots,\bsu_{b^m-1}\}$ in $[0,1)^s$ is called
a $(t,m,s)$-net in base $b$ if every interval of the form $\prod_{j=1}^s\left[\frac{a_j}{b^{k_j}},\frac{a_j+1}{b^{k_j}}\right)$ contains exactly $b^t$ points of $\cp$ for all integers $a_j\in[0,b^{k_j})$ and all $k_j\in \natu_0$ satisfying $\sum_{j=1}^sk_j=m-t$. For $\emptyset \neq w\subseteq 1{:}s$, the projection of $\cp$ on  coordinates $j\in w$ forms a $(t_w,m,|w|)$-net in base $b$, where $t_\omega\in \natu$ is called $t$-quality parameter and $t_{w}=t$ when $w=1{:}s$. When performing Owen's scrambling \cite{rtms} on $\cp$, the resulting point set is called the scrambled $(t,m,s)$-net, which is also a $(t,m,s)$-net with probability one.

Following \cite{pan:2025}, we assume that the density $\varphi (x)$  is a strictly  positive, bounded, symmetric, light-tailed  function in the following.

\begin{assumption}~\label{assump:rho}
Assume that $\varphi(x)>0$, $\varphi(x)=\varphi(-x)$, $\varphi_\infty:=\sup_{x\in \real} \varphi(x)<\infty$, and for any $\varepsilon>0$, there exists a constant $c_\varepsilon>0$ such that
$\varphi(x)\geq c_\varepsilon \Phi(x)^{1+\varepsilon}$ for any $x\le 0$. 
\end{assumption}

Recall that the resulting integrand over the unit cube $\itr^s$ under the IS is $f^w(\bsu)=w(\bsu) f\circ T(\bsu)$. Define
\begin{align}
    \Vert f\Vert_{q}&=\sup_{v\subseteq 1{:}s}\Vert f_v \Vert_{W^{1,q}_{\mix}(\real^s,\varphi)}\label{eq:fq},\\
    \gamma_v&=\begin{cases}
        \Vert f\Vert^{-1}_{q} \Vert f_v \Vert_{W^{1,q}_{\mix}(\real^s,\varphi)}\quad &\Vert f\Vert_{q}>0\\
        0\quad &\Vert f\Vert_{q}=0
    \end{cases}.\label{eq:gammav}
\end{align}
Note that $\gamma_v$ can be viewed as an adjusted first-order mixed $L^q$ norm of the ANOVA term $f_v$ to satisfy $\gamma_v\in[0,1]$.
We are ready to state our main theorem, as a generalization of  \cite[Theorem 5.3]{pan:2025}. Its proof is encapsulated in Section~\ref{sec:proofs} for interested readers.

\begin{theorem}\label{thm:qmcvar}
    Suppose that Assumption~\ref{assump:rho} is satisfied. Let $\{\bsu_0,\dots,\bsu_{n-1}\}$ be a scrambled digital net in base $b\ge 2$ with $n=b^m$ and $t$-quality parameters $\{t_\omega\mid \emptyset\neq \omega\subseteq 1{:}s\}$. Then given $\alpha\in (0,1)$, $f\in W^{1,q}_{\mix}(\real^s,\varphi)$ for $q>1$, $w_j=w_{\theta_j}$  for $\{\theta_j\mid j\in 1{:}s\}\subseteq (0,1/2]$ and $\varepsilon>0$, we have
    \begin{equation}\label{eqn:maintheorem}
       \e \Big|\frac{1}{n}\sum_{i=0}^{n-1}f^w(\bsu_i)-\int_{\real^s} f(\bsx) \prod_{j=1}^s \varphi(x_j)\rd \bsx\Big|^2 
            \leq \frac{\Vert f\Vert^2_q }{n^{1+\alpha}}\sum_{\emptyset \neq \omega\subseteq 1{:}s }  C^{|\omega|}_{*} m^{|\omega|-1}\Tilde{\gamma}_\omega,  
    \end{equation}
        where $C_{*}$ is a constant depending on $\varepsilon,p,q,b$ and $\alpha$, and
        $$\Tilde{\gamma}_\omega=b^{(1+\alpha)t_\omega} \sum_{v\subseteq \omega}\gamma^2_v\prod_{j\in v} \theta^{-\nu_{\alpha,q}-\varepsilon}_j \prod_{j\in \omega\setminus v} \theta^{1-\alpha}_j$$
        for $\gamma_v$ given by \eqref{eq:gammav} and
        $$
        \nu_{\alpha,q}=\max(2\alpha/q,2/q+\alpha-1)=\begin{cases}
            2/q+\alpha-1 & q\in(1,2]\\
            2\alpha/q & q\in(2,\infty).
        \end{cases}
        $$
    \end{theorem}

\begin{remark}
Under Assumption~\ref{assump:rho},  the CDF $\Phi(x)$ and  its inverse $\Phi^{-1}(u)$ is almost everywhere differentiable and strictly increasing. It is easy to verify the Gaussian density $\varphi(x)=\exp(-x^2/2)/\sqrt{2\pi}$ satisfies Assumption~\ref{assump:rho}.    As shown in \cite{c812a4a5-75ac-3614-936d-5782e44941d0}, for the Sobol' sequence \cite{sobol67} and the Niederreiter sequence \cite{NIEDERREITER198851}, the $t$-quality parameters $t_w$ satisfy that
\begin{equation}\label{eqn:tutj}
    t_\omega\leq \sum_{j\in \omega}t_j
\end{equation}
with $t_j= O(\log_b(j))$, which is an important ingredient for analyzing dimension-independent convergences in the following corollary.
\end{remark}

\begin{corollary}\label{cor:qmcvar}
        Suppose that the $t$-quality parameters $\{t_\omega\mid \emptyset\neq \omega\subseteq 1{:}s\}$ of the  scrambled digital net  satisfy \eqref{eqn:tutj} and the ANOVA terms of the function $f(\bsx)$ have   adjusted first-order mixed $L^q$ norms 
       satisfying
        $$\gamma_v\leq \prod_{j\in v}\Gamma_j \ \forall v\subseteq 1{:}s$$
        for $\{\Gamma_j\mid j\in 1{:}s\}$. 
        Then under the assumptions of Theorem~\ref{thm:qmcvar},  for any $\lambda\in (0,1)$,
        $$\e \Big|\frac{1}{n}\sum_{i=0}^{n-1}f^w(\bsu_i)-\int_{\real^s} f(\bsx) \prod_{j=1}^s \varphi(x_j)\rd \bsx\Big|^2 \leq \frac{\Vert f\Vert^2_q }{n^{1+\alpha}}\exp(C_{\lambda}S_\lambda m^\lambda),$$
        where $n=b^m$, $C_{\lambda}$ is a constant depending on $\lambda$ and 
        $$S_\lambda=\sum_{j=1}^s \Tilde{\Gamma}^\lambda_j \quad \text{for} \quad \Tilde{\Gamma}_j=C_*b^{(1+\alpha)t_j} (\Gamma^2_j \theta^{-\nu_{\alpha,q}-\varepsilon }_j+\theta^{1-\alpha}_j).$$
    \end{corollary}
    \begin{proof}
        Plugging in the upper bounds for $t_\omega$ and $\gamma_v$ into equation~\eqref{eqn:maintheorem}, a straightforward calculation shows 
        $$\e \Big|\frac{1}{n}\sum_{i=0}^{n-1}f^w(\bsu_i)-\int_{\real^s} f(\bsx) \prod_{j=1}^s \varphi(x_j)\rd \bsx\Big|^2 \leq \frac{\Vert f\Vert^2_q }{n^{1+\alpha}}\sum_{\emptyset\neq \omega \subseteq 1{:}s}m^{|\omega|-1}\prod_{j\in \omega} \Tilde{\Gamma}_j.$$
        Because $x^{-\lambda}\log(1+x)\to 0$ when $x\to 0^+$ or $x\to\infty$, we can find a constant $C_\lambda$ so that $\log(1+x)\leq C_\lambda x^{\lambda}$ over $x\in (0,\infty)$. Therefore,
        $$\sum_{\emptyset\neq \omega \subseteq 1{:}s}m^{|\omega|-1}\prod_{j\in \omega} \Tilde{\Gamma}_j\leq \sum_{\omega \subseteq 1{:}s}m^{|\omega|}\prod_{j\in \omega} \Tilde{\Gamma}_j=\prod_{j=1}^s (1+m \Tilde{\Gamma}_j)\leq \exp\Big(C_{\lambda}\sum_{j=1}^sm^\lambda \Tilde{\Gamma}^\lambda_j\Big).$$
        This completes the proof.
 \end{proof}
 \begin{remark}\label{rmk:tractability}
Since $\lim_{n\to\infty} n^{-\varepsilon}\exp(m^\lambda)=0$ for any $\varepsilon>0$, the above corollary shows $\hat{\mu}$ attains dimension-independent MSE rate of nearly $O(n^{-1-\alpha})$ if $S_\lambda$ and $\Vert f\Vert^2_q $ stay bounded as $s\to\infty$ for some $\lambda\in(0,1)$. It generalizes \cite[Corollary 5.5]{pan:2025} to cases where  $f\in W^{1,q}_{\mix}(\real^s,\varphi)$ for $q\in (1,\infty)$.
\end{remark}

\begin{remark}\label{rmk:optimalalpha}
% By a similar analysis as in \cite[Remark 5.6]{pan:2025}, when $\Gamma_j=O(j^{-\rho})$ for $\rho>\max(2/q,1)$ and $t_j=O(\log_b(j))$, we can set
% $$\theta_j=\begin{cases}
%     \theta_0 j^{-\rho q}, &\text{ if } q\in (1,2] \\
%     \theta_0 j^{-\rho q'} \text{ for } q'\in (2,q] , &\text{ if } q\in (2,\infty)
% \end{cases}$$
% with $\theta_0\in (0,1/2]$, and show that the MSE of $\hat{\mu}$ converges at a nearly $O(n^{-1-\alpha^*})$ rate independent of the dimension, where
% $$\alpha^*=\begin{dcases}
%     \frac{\rho q -2}{\rho q+1}, &\text{ if } q\in (1,2] \\
%     \frac{\rho q' -2}{\rho q'+1}, &\text{ if } q\in (2,\infty) \text{ and } \frac{\rho q' -2}{\rho q'+1}\geq \frac{q(q' -2)}{q'(q-2)}\\
%     \frac{2\rho q-2q}{2\rho q'+q}, &\text{ if } q\in (2,\infty) \text{ and } \frac{\rho q' -2}{\rho q'+1}< \frac{q(q' -2)}{q'(q-2)}
% \end{dcases}.$$
% Note that in the third case the rate improves when we decrease $q'$. This is because $\Gamma^2_j \theta^{-\nu_{\alpha,q}-\varepsilon }_j$, which originates from the $W^{1,1}_{\mix}(\itr^s)$-norm of $f^w$, is dominating in this case and a smaller $q'$ produces less singular derivatives.

Assume that $\Gamma_j=O(j^{-\rho})$ for $\rho>\max(2/q,1)$ and $t_j=O(\log_b(j))$. Then
$$
\Tilde{\Gamma}_j=O(j^{1+\alpha} (j^{-2\rho} \theta^{-\nu_{\alpha,q}-\varepsilon }_j+\theta^{1-\alpha}_j)).
$$
Following an analysis similar to that in \cite[Remark 5.6]{pan:2025}, we balance the two terms $j^{-2\rho} \theta^{-\nu_{\alpha,q}}_j$ and $\theta^{1-\alpha}_j$ in $\Tilde{\Gamma}_j$ by setting
\begin{equation}\label{eqn:thetachoice}
 \theta_j=\begin{cases}
    \theta_0 j^{-\rho q}, &\text{ if } q\in (1,2] \\
    \theta_0 j^{-2\rho/(2\alpha/q+1-\alpha)}, &\text{ if } q\in (2,\infty)
\end{cases}   
\end{equation}
with a constant $\theta_0\in (0,1/2]$.

For $q\in (1,2]$, the above choice gives $\Tilde{\Gamma}_j=O(j^{1+\alpha-(1-\alpha)\rho q+\rho q \varepsilon})$. To ensure $\sum_{j=1}^\infty \Tilde{\Gamma}^\lambda_j<\infty$ for some $\lambda\in (0,1)$, we require that 
$$
1+\alpha-(1-\alpha)\rho q<-1,
$$
since $\varepsilon>0$ may be taken arbitrarily small.
This leads to the condition 
$$0<\alpha <a^*(q,\rho):=\frac{\rho q-2}{\rho q+1}.$$ 
For $q>2$, we have 
$
\Tilde{\Gamma}_j=O(j^{1+\alpha-2\rho(1-\alpha)/(2\alpha/q+1-\alpha)+\varepsilon'}),
$
where $\varepsilon'=2\varepsilon\rho/(2\alpha/q+1-\alpha)>0$ is arbitrarily small. To ensure $\sum_{j=1}^\infty \Tilde{\Gamma}^\lambda_j<\infty$ for some $\lambda\in (0,1)$, we need
$$
1+\alpha-2\rho(1-\alpha)/(2\alpha/q+1-\alpha)<-1,
$$
or equivalently
$$
F(\alpha):=(q-2)\alpha^2-(2q\rho-q+4)\alpha+2q(\rho-1)>0.
$$
Since $F(0)>0$ and $F(1)<0$,  we have 
$$
0<\alpha<\alpha^*(q,\rho) := 1-\frac{\sqrt{((2\rho-3)q+8)^2+24(q-2)}-(2\rho-3)q-8}{2(q-2)}<1.
$$
In conclusion, under the parameter choice \eqref{eqn:thetachoice}, the MSE of $\hat{\mu}$ converges at a nearly $O(n^{-1-\alpha^*(q,\rho)})$ rate with an implied constant independently of the dimension $s$, where
\begin{equation}\label{eq:alphastar}
    \alpha^*(q,\rho)=\begin{dcases}
    \frac{\rho q -2}{\rho q+1}, &\text{ if } q\in (1,2] \\
    1-\frac{\sqrt{((2\rho-3)q+8)^2+24(q-2)}-(2\rho-3)q-8}{2(q-2)}, &\text{ if } q>2
\end{dcases}.
\end{equation}
\end{remark}

\begin{remark}\label{rem:qinfinity}
If $\rho\ge 3/2$, it is clear that
$$
\lim_{q\to \infty} \alpha^*(q,\rho) = 1.
$$
If $\rho\in (1, 3/2)$, we instead have
$$
\lim_{q\to \infty} \alpha^*(q,\rho) = 2\rho-2.
$$
Suppose that $f$ has no significant boundary growth in the sense that $f\in W^{1,q}_{\mix}(\real^s,\varphi)$ for every $q\in (1,\infty)$. If $\rho\ge 3/2$, we should set $\theta_j=0$ and use the usual inversion methods, consistent with the fact that inversion methods attain nearly $O(n^{-2})$ MSEs for such integrands \cite{dick2025quasi}. Otherwise, if $\rho\in (1, 3/2)$, we take $\theta_j=\theta_0 j^{-2\rho/(3-2\rho)}$, attaining nearly $O(n^{-2\rho+1})$ MSE. 
To unify two cases, we conclude that our BDIS attains nearly $O(n^{-1-\min\{2\rho-2,1\}})$ dimension-independent MSE if $f\in W^{1,q}_{\mix}(\real^s,\varphi)$ for every $q\in (1,\infty)$.
    % If $f$ has no significant boundary growth in the sense that $f\in W^{1,q}_{\mix}(\real^s,\varphi)$ for every $q\in (1,\infty)$ and $\Gamma_j=O(j^{-3/2})$ for all sufficiently large $q$, Remark~\ref{rmk:optimalalpha} suggests we should set $\theta_j=0$ and use the usual inversion methods, consistent with the fact that inversion methods attain nearly $O(n^{-2})$ MSEs for such integrands \cite{dick2025quasi}.
\end{remark}

%\subsection{Walsh functions and Walsh decomposition}\label{subsec:walsh}

\section{Proof of Theorem~\ref{thm:qmcvar}}\label{sec:proofs}
We first present some preliminaries for the proof. Assume that $b$ is a prime number. For non-negative integer $k<b^r$ and $u\in[0,1)$, we write their  $b$-adic expansions as $k=\sum_{i=1}^r \kappa_i b^{i-1}$ and $u=\sum_{i=1}^\infty v_i b^{-i}$, respectively.
The $k$-th $b$-adic Walsh function is given by
\begin{equation}\label{eqn:walk}
    {}_b\walk (u)=\exp\left(\frac{2\pi \ii}{b}\sum_{i=1}^r \kappa_i v_i\right),
\end{equation}
where $\ii=\sqrt{-1}$. For the multivariate case, the Walsh functions are naturally defined by the product form
$$
{}_b\walbsk (\bsu) :=\prod_{j=1}^s {}_b\walkj(u_j),\quad \bsk\in \natu_0^s.
$$
Note that $\{{}_b\walbsk (\bsu)\mid \bsk\in \natu_0^s\}$ forms an orthonormal basis of $L^2(\itr^s)$ \cite{dick:pill:2010}. For $f\in L^2(\itr^s)$, we  have the Walsh series expanding
$$
f(\bsu)\sim \sum_{\bsk\in\natu^s} \hat{f}(\bsk){}_b\walbsk (\bsu),
$$
where $f\sim g$ denotes the equivalence relation for the $L^2(\itr^s)$ space, and 
$$\hat{f}(\bsk):=\int_{\itr^s}f(\bsu)\overline{{}_b\walbsk (\bsu)} \rd \bsu$$ denotes the $\bsk$-th Walsh coefficient of $f$.

For $k\in \natu_0$, we denote $\lceil k\rceil=r$ if $k=\sum_{i=1}^r \kappa_i b^{i-1}$ with $\kappa_r\neq 0$ and $\lceil k\rceil=0$ if $k=0$. For $\bsk\in \natu^s_0$, we denote $\supp(\bsk)=\{j\in 1{:}s\mid k_j\neq 0\}$ and $\lceil \bsk\rceil=(\lceil k_1\rceil,\dots,\lceil k_s\rceil)$. For $\bsell\in \natu_0^s$, we let
\begin{equation*}
  \sigma^2_{\bsell}=\sum_{\bsk\in L_{\bsell}} |\hat{f}(\bsk)|^2 \quad\text{for}\quad  L_{\bsell}=\{\bsk\in \natu_0^s\mid \lceil \bsk\rceil=\bsell\}.
\end{equation*}
The next lemma comes from \cite[Lemma 2.7]{pan:2025}, which  actually collected the results in \cite{dick:pill:2010,GODA2023101722,owen96,snxs}.

\begin{lemma}\label{lem:qmcvar}
For $f\in L^2(\itr^s)$ and $\{\bsu_0,\dots,\bsu_{n-1}\}$ a scrambled  $(t,m,s)$-net in base $b\ge 2$ with $n=b^m$,
\begin{equation}\label{eqn:RQMCvar}
    \e\left|\frac{1}{n}\sum_{i=0}^{n-1} f(\bsu_i)-\int_{\itr^s} f(\bsu)\rd \bsu\right|^2=\frac{1}{n}\sum_{\emptyset\neq \omega \subseteq 1{:}s}\sum_{\bsell\in \natu^\omega} \Gamma_{\omega,\bsell} \sigma^2_{\bsell},
\end{equation}
    where $\natu^{v}=\{\bsell\in \natu^s_0\mid \supp(\bsell)=v\}$ and 
    \begin{equation}\label{eqn:gaincoefbound}
        \Gamma_{\omega,\bsell}\leq \Big(\frac{b}{b-1}\Big)^{|\omega|-1}b^{t_\omega}\bsone\{\Vert \bsell\Vert_1> m-t_\omega-|\omega|\}.
    \end{equation}
\end{lemma}
% \begin{proof}
%     Equation~\eqref{eqn:RQMCvar} is first derived by \cite{owen96,snxs} in terms of Haar wavelet basis. See \cite[Theorem 13.6]{dick:pill:2010} for the version using Walsh basis. Inequality~\eqref{eqn:gaincoefbound} comes from \cite{GODA2023101722}.
% \end{proof}
    
We next divide the proof of Theorem~\ref{thm:qmcvar} into three steps. First, we bound the $W^{1,1}_{\mix}(\itr^s)$-norm of $f^w$ when $f\in W^{1,q}_{\mix}(\real^s,\varphi)$ for $q>1$. Next, we bound the variance component $\sigma^2_{\bsell}$ for $f^w$. Finally, we apply Lemma~\ref{lem:qmcvar} to prove the main theorem.

\subsection{Step I: Norms of $f^w$}

The next theorem is crucial for characterizing the dependence of $\Vert f^w \Vert_{W^{1,1}_{\mix}(\itr^s)}$ on $\theta_j,j\in 1{:}s$.

\begin{theorem}\label{thm:Wqcase}
    Given $\varphi$ satisfying Assumption~\ref{assump:rho}, $f\in W^{1,q}_{\mix}(\real^s,\varphi)$ for $q>1$, $w_j=w_{\theta_j}$  for $\{\theta_j\mid j\in 1{:}s\}\subseteq (0,1/2]$ and $\varepsilon>0$, we have
    $$\Vert f^w \Vert_{W^{1,1}_{\mix}(\itr^s)}\leq  \Vert f \Vert_{W^{1,q}_{\mix}(\real^s,\varphi)}\prod_{j=1}^sC_{\varepsilon,p,q}\theta^{-1/q-\varepsilon}_j,$$
    where $C_{\varepsilon,p,q}$ is a constant depending on $\varepsilon,p$ and $q$.
\end{theorem}

\begin{proof}
    A straightforward calculation shows
   \begin{align*}
        \partial^{v} f^w (\bsu)= \Big(\prod_{j\in 1{:}s\setminus v} w_j(u_j)\Big)\sum_{v'\subseteq v}(\partial^{v'} f) \circ T(\bsu) \prod_{j\in v'} \frac{w_j(u_j)^{2}}{\varphi\circ T_j(u_j)}\prod_{j\in v\setminus v'} w'_j(u_j).
    \end{align*} 
    By Hölder's inequality and $\sup_{u\in \itr} w_\theta(u)\leq 2$ for $\theta\in (0,1/2]$,
    \begin{align}\label{eqn:I1I2}
   &\Vert\partial^{v} f^w \Vert_{L^1(\itr^s)} \\
    \leq & 2^{s-|v|}\sum_{v'\subseteq v}\int_{\itr^{s}}|(\partial^{v'} f) \circ T(\bsu)| w(\bsu)^{1/q}\prod_{j\in v'} \frac{w_j(u_j)^{2-1/q}}{\varphi\circ T_j(u_j)}\prod_{j\in v\setminus v'} \frac{|w'_j(u_j)|}{w_j(u_j)^{1/q}} \rd \bsu \nonumber\\
   \leq & 2^{s-|v|}\sum_{v'\subseteq v}\Big(\int_{\itr^s}|(\partial^{v'} f) \circ T(\bsu)|^q w(\bsu)\rd \bsu\Big)^{1/q} \prod_{j\in v'} I_1(\theta_j,q)\prod_{j\in v\setminus v'} I_2(\theta_j,q) \nonumber,
   \end{align}
   where, by denoting $q^*=q/(q-1)$,
   $$I_1(\theta,q)=\Big(\int_{\itr}\frac{w_\theta(u)^{q^*+1}}{(\varphi\circ T_\theta(u))^{q^*}}\rd u\Big)^{1/q^*}$$
   and
   $$ I_2(\theta,q)=\Big(\int_{\itr}\frac{|w'_\theta(u)|^{q^*}}{w_\theta(u)^{q^*-1}}\rd u\Big)^{1/q^*}.$$
   First, we bound $I_1(\theta,q)$. By \cite[Lemma 4.3]{pan:2025}, for $u\in (0,\theta)$,
  $$\varphi\circ T_\theta(u)\geq 
       c_{\varepsilon,p} \Big(  \theta w_\theta(u)\min\big((2u/\theta)^{p+1},1 \big)\Big)^{1+\varepsilon}\geq c_{\varepsilon,p} \Big(\theta w_\theta(u) (u/\theta)^{p+1} \Big)^{1+\varepsilon},$$
where $c_{\varepsilon,p}$ is a constant depending on $\varepsilon$ and $p$. When $u\in [\theta,1/2]$, we use Assumption~\ref{assump:rho} to bound
       $$\varphi\circ T_\theta(u)\geq c_\varepsilon\Big(\Phi\circ\Phi^{-1}\big(\int_0^u w_\theta(u')\rd u'\big) \Big)^{1+\varepsilon}=c_\varepsilon(u-\theta/2)^{1+\varepsilon}.$$
       Therefore, 
       \begin{align*}
       I_1(\theta,q)=&\Big(2\int_0^{1/2}\frac{w_\theta(u)^{q^*+1}}{(\varphi\circ T_\theta(u))^{q^*}}\rd u\Big)^{1/q^*}\\
       \leq & \Big(2\int_{0}^{\theta}\frac{w_\theta(u)^{q^*+1}}{ c^{q^*}_{\varepsilon,p}(\theta w_\theta(u) (u/\theta)^{p+1})^{q^*(1+\varepsilon)}}\rd u+2\int_\theta^{1/2} \frac{(1-\theta)^{-q^*-1}}{c^{q^*}_\varepsilon (u-\theta/2)^{q^*(1+\varepsilon)}}\rd u\Big)^{1/q^*}\\
       \leq &\Big(2\theta^{1-q^*(1+\varepsilon)}\int_{0}^{1}\frac{w_\theta(\theta u)^{1-q^*\varepsilon}}{ c^{q^*}_{\varepsilon,p}  u^{q^*(p+1)(1+\varepsilon)}}\rd u+2\int_{\theta/2}^{\infty} \frac{2^{q^*+1}}{c^{q^*}_\varepsilon u^{q^*(1+\varepsilon)}}\rd u\Big)^{1/q^*}\\  
       \leq & C_{1} \theta^{-1/q-\varepsilon},   
       \end{align*}
       where $C_1$ depends on $\varepsilon,p$ and $q$, and we have used $w_\theta(\theta u)= 2^{-p-2} u^{-p-1}\exp(2^p-u^{-p})$ over $u\in (0,1/2)$ in the last inequality.
       
       Next, we bound $I_2(\theta,q)$. By \cite[Lemma 4.4]{pan:2025}, 
       $$|w'_\theta(u)|\leq c_p\theta^{-1}w_\theta(u)$$ for $u\in (0,\theta]$ and a constant $c_p$ depending on $p$. Hence,
       \begin{align*}
       I_2(\theta,q) = \Big(2\int_{0}^\theta\frac{|w'_\theta(u)|^{q^*}}{w_\theta(u)^{q^*-1}}\rd u\Big)^{1/q^*}
       \leq & \Big(2c_p^{q^*}\theta^{-q^*}\int_{0}^\theta w_\theta(u)\rd u\Big)^{1/q^*}\\
       = & \Big(2c_p^{q^*}\theta^{-q^*+1}\int_{0}^1 w_\theta(\theta u)\rd u\Big)^{1/q^*}
       \leq   C_2 \theta^{-1/q},
       \end{align*}
       where $C_2$ depends on $p$ and $q$. Putting the above bounds into equation~\eqref{eqn:I1I2} and substituting
       $$\int_{\itr^s}|(\partial^{v'} f) \circ T(\bsu)|^q w(\bsu)\rd \bsu=\int_{\itr^s}|(\partial^{v'} f) \circ T(\bsu)|^q\prod_{j=1}^s \varphi\circ T_j(u_j)  \rd T(\bsu)=\Vert \partial^{v'} f\Vert^q_{L^q(\real^s,\varphi)},$$
       we can bound
              \begin{align*}
        &\Vert\partial^{v} f^w \Vert_{L^1(\itr^s)}\\ 
    \leq & 2^{s-|v|}\sum_{v'\subseteq v}\Vert \partial^{v'} f\Vert_{L^q(\real^s,\varphi)} \prod_{j\in v'} C_{1} \theta_j^{-1/q-\varepsilon} \prod_{j\in v\setminus v'} C_2 \theta_j^{-1/q}    \\
    \leq  & 2^{s-|v|}\Big(\sum_{v'\subseteq v}\Vert \partial^{v'} f\Vert^q_{L^q(\real^s,\varphi)} \Big)^{1/q}\Big(\sum_{v'\subseteq v}\prod_{j\in v'} C^{q^*}_{1}\prod_{j\in v\setminus v'}C^{q^*}_{2}\Big)^{1/q^*}\prod_{j\in v} \theta_j^{-1/q-\varepsilon} \\
    \leq & 2^{s-|v|}\Vert f \Vert_{W^{1,q}_{\mix}(\real^s,\varphi)}\prod_{j\in v} (C^{q^*}_{1}+C^{q^*}_{2})^{1/q^*}\theta_j^{-1/q-\varepsilon} .
       \end{align*}
       Finally,
       \begin{align*}
        \Vert f^w \Vert_{W^{1,1}_{\mix}(\itr^s)}=&\sum_{v\subseteq 1{:}s}\Vert\partial^{v} f^w \Vert_{L^1(\itr^s)}  \\
        \leq & \Vert f \Vert_{W^{1,q}_{\mix}(\real^s,\varphi)} \sum_{v\subseteq 1{:}s}2^{s-|v|}\prod_{j\in v} (C^{q^*}_{1}+C^{q^*}_{2})^{1/q^*}\theta_j^{-1/q-\varepsilon} \\
        \leq &  \Vert f \Vert_{W^{1,q}_{\mix}(\real^s,\varphi)} \prod_{j=1}^s\Big((C^{q^*}_{1}+C^{q^*}_{2})^{1/q^*}+2\Big) \theta_j^{-1/q-\varepsilon}.
       \end{align*}
\end{proof}

\begin{remark}
    We can further generalize Theorem~\ref{thm:Wqcase} by using \cite[Theorem 4.5]{pan:2025} to bound $\Vert f^w \Vert_{W^{1,q}_{\mix}(\itr^s)}$  and applying the interpolation inequality \cite[Theorem 2.11]{adams2003sobolev} to bound $\Vert f^w \Vert_{W^{1,q'}_{\mix}(\itr^s)}$ for $q'\in [1,q]$.
\end{remark}

\subsection{Step II: Bounds on $\sigma^2_{\bsell}$}

We need the following two lemmas:

\begin{lemma}\label{lem:fkbound}
For $\bsk\in \natu^s$ and $g\in W^{1,1}_{\mix}(\itr^s)$,
$$|\hat{g}(\bsk)|\leq  \Vert\partial^{1{:}s}g\Vert_{L^1(\itr^s)}\prod_{j=1}^s b^{-\lceil k_j\rceil+1}.$$
\end{lemma}
\begin{proof}
By \cite[proof of Lemma 3.3]{pan:2025},
$$\hat{g}(\bsk)=\int_{\itr^s}\partial^{1{:}s}g(\bsu) \prod_{j=1}^s \overline{ \mathcal{I}({}_b\walkj)(u_j)} \rd \bsu, $$
where 
$$\mathcal{I}({}_b\walk)(u)=\int_0^u {}_b\walk (t)\rd t.$$
As shown in \cite[Example 3.4]{pan:2025}, if $\kappa_r b^{r-1}\leq k < (\kappa_r+1) b^{r-1}$ for $r=\lceil k\rceil$, 
$$\mathcal{I}({}_b\walk)={}_b\mathrm{wal}_{\mathcal{M}(k)} \mathcal{I}({}_b\mathrm{wal}_{\mathcal{N}(k)})$$
for $\mathcal{M}(k)=k-\kappa_r b^{r-1}$ and $\mathcal{N}(k)=\kappa_r b^{r-1}$, and 
$\Vert \mathcal{I}({}_b\mathrm{wal}_{\mathcal{N}(k)})\Vert_{L^\infty(\itr)}\leq b^{-r+1}.$ 
Finally, we apply Hölder's inequality and conclude
$$|\hat{g}(\bsk)|\leq \Vert\partial^{1{:}s}g\Vert_{L^1(\itr^s)} \prod_{j=1}^s \Vert \mathcal{I}({}_b\walkj)(u_j)\Vert_{L^\infty(\itr)}\leq \Vert\partial^{1{:}s}g\Vert_{L^1(\itr^s)} \prod_{j=1}^s b^{-\lceil k_j\rceil+1}. $$
\end{proof}

\begin{remark}
    A similar bound is derived in \cite{SUZUKI20161} given $\partial^{1{:}s}g$ is continuous. Our lemma only requires $\partial^{1{:}s}g$ to exist weakly.
\end{remark}

\begin{lemma}\label{lem:sigmabound}
   For $\alpha\in (0,1)$, $\bsell\in \natu^s$ and $g\in W^{1,1}_{\mix}(\itr^s)$,
   $$\sum_{\bsk\in L_{\bsell}} |\hat{g}(\bsk)|^2 \leq C^s_{b,\alpha} b^{-\alpha\Vert \bsell\Vert_1}\Vert g \Vert^{2-2\alpha}_{L^2(\itr^s)}\Vert\partial^{1{:}s}g\Vert^{2\alpha}_{L^1(\itr^s)}, $$
   where $L_{\bsell}=\{\bsk\in \natu_0^s\mid \lceil \bsk\rceil=\bsell\}$ and $C_{b,\alpha}$ is a constant depending on $b$ and $\alpha$.
\end{lemma}
\begin{proof}
    We apply Lemma~\ref{lem:fkbound} to each $\bsk\in L_{\bsell}$ and obtain
    $$|\hat{g}(\bsk)|^2\leq b^{-2\Vert \bsell\Vert_1+2s}\Vert\partial^{1{:}s}g\Vert^2_{L^1(\itr^s)}.$$
    Because $|L_{\bsell}|=\prod_{j=1}^s (b-1)b^{\ell_j-1}$,
     $$\sum_{\bsk\in L_{\bsell}} |\hat{g}(\bsk)|^2\leq |L_{\bsell}|b^{-2\Vert \bsell\Vert_1+2s}\Vert\partial^{1{:}s}g\Vert^2_{L^1(\itr^s)}= b^s (b-1)^sb^{-\Vert \bsell\Vert_1}\Vert\partial^{1{:}s}g\Vert^2_{L^1(\itr^s)}.$$
     Meanwhile, $g\in L^2(\itr^s)$ by Lemma~\ref{lem:Sobolevimbedding} and $g\in W^{1,1}_{\mix}(\itr^s)$. From Bessel's inequality,
     $$\sum_{\bsk\in L_{\bsell}} |\hat{g}(\bsk)|^2\leq \Vert g \Vert^2_{L^2(\itr^s)}.$$
The conclusion follows by combining the two bounds.
\end{proof}

\begin{theorem}\label{thm:sigmabound}
Given  $\alpha\in (0,1)$, $\bsell\in \natu_0^{s}\setminus\{\bszero\}$, $f\in W^{1,q}_{\mix}(\real^s,\varphi)$ for $q>1$, $w_j=w_{\theta_j}$  for $\{\theta_j\mid j\in 1{:}s\}\subseteq (0,1/2]$ and $\varepsilon>0$, we have
$$       \sigma^2_{\bsell}=\sum_{\bsk\in L_{\bsell}} |\hat{f}^w(\bsk)|^2
        \leq 
        C_{\varepsilon,p,q,b,\alpha}^{|\supp(\bsell)|} b^{-\alpha\Vert \bsell\Vert_1}\sum_{v\subseteq\supp(\bsell)}\Vert f_v \Vert^2_{W^{1,q}_{\mix}(\real^s,\varphi)}\prod_{j\in v} \theta^{-\nu_{\alpha,q}-\varepsilon}_j\prod_{j\in \supp(\bsell)\setminus v} \theta^{1-\alpha}_j,$$
    where $L_{\bsell}=\{\bsk\in \natu_0^s\mid \lceil \bsk\rceil=\bsell\}$, $\nu_{\alpha,q}=\max(2\alpha/q,2/q+\alpha-1)$ and  $C_{\varepsilon,p,q,b,\alpha}$ is a constant depending on $\varepsilon,p,q, b$ and $\alpha$.
\end{theorem}
\begin{proof}
    From \cite[Lemma 3.1]{pan:2025},
$$\hat{f}^w(\bsk)=\sum_{v\subseteq\supp(\bsk)}\hat{f}^w_v(\bsk_v)\prod_{j\in \supp(\bsk)\setminus v} \hat{w}_j(k_j) \quad \text{for} \quad f^w_v(\bsu_v)=f_v\circ T(\bsu) \prod_{j\in v} w_j(u_j)$$
with $f_v\circ T(\bsu)$ interpreted as a function of $\bsu_v$.
Using Cauchy–Schwarz inequality,
\begin{align*}
 \sigma^2_{\bsell}=\sum_{\bsk\in L_{\bsell}} |\hat{f}^w(\bsk)|^2\leq &\sum_{\bsk\in L_{\bsell}} 2^{|\supp(\bsell)|} \sum_{v\subseteq\supp(\bsk)}|\hat{f}^w_v(\bsk_v)|^2\prod_{j\in \supp(\bsk)\setminus v} |\hat{w}_j(k_j)|^2   \\
 =&  2^{|\supp(\bsell)|} \sum_{v\subseteq\supp(\bsell)} \Big(\sum_{\bsk_v\in L_{\bsell,v}}|\hat{f}^w_v(\bsk_v)|^2\Big)\prod_{j\in \supp(\bsell)\setminus v} \sum_{k_j=b^{\ell_j-1}}^{b^{\ell_j}-1}|\hat{w}_j(k_j)|^2 ,
\end{align*}
where $L_{\bsell,v}=\{\bsk_v\in \natu^{v}\mid \lceil \bsk_v\rceil =\bsell_v\}$.
   By identifying $v$ with $1:|v|$, Lemma~\ref{lem:sigmabound} implies
   \begin{equation}\label{eqn:kvsum}
\sum_{\bsk_v\in L_{\bsell,v}}|\hat{f}^w_v(\bsk_v)|^2\leq C^{|v|}_{b,\alpha} b^{-\alpha\sum_{j\in v}\ell_j}\Vert f^w_v \Vert^{2-2\alpha}_{L^2(\itr^{|v|})}\Vert\partial^{v}f^w_v\Vert^{2\alpha}_{L^1(\itr^{|v|})}.       
   \end{equation} 
Theorem~\ref{thm:Wqcase} shows
$$\Vert\partial^{v}f^w_v\Vert_{L^1(\itr^{|v|})}\leq\Vert f^w_v \Vert_{W^{1,1}_{\mix}(\itr^{|v|})}\leq \Vert f_v \Vert_{W^{1,q}_{\mix}(\real^s,\varphi)}  \prod_{j\in v}C_{\varepsilon,p,q}\theta^{-1/q-\varepsilon}_j.$$
If $q\geq 2$, we use \cite[Lemma 3.1]{pan:2025} to bound
$$ \Vert f^w_v \Vert_{L^2(\itr^{|v|})}\leq \Vert f^w_v \Vert_{L^q(\itr^{|v|})}\leq  2^{\frac{q-1}{q}|v|}\Vert f_v\Vert_{L^q(\real^{s},\varphi)}\leq 2^{\frac{q-1}{q}|v|}\Vert f_v \Vert_{W^{1,q}_{\mix}(\real^s,\varphi)} .$$
If $q\in (1,2)$, we use Lemma~\ref{lem:Sobolevimbedding} and the interpolation inequality \cite[Theorem 2.11]{adams2003sobolev} to bound
\begin{align*}
  \Vert f^w_v \Vert_{L^2(\itr^{|v|})}\leq & \Vert f^w_v \Vert^{q/2}_{L^q(\itr^{|v|})}\Vert f^w_v \Vert^{1-q/2}_{L^\infty(\itr^{|v|})}  \\
  \leq & 2^{\frac{q-1}{2}|v|}\Vert f_v\Vert^{q/2}_{L^q(\real^{s},\varphi)} \Vert f^w_v \Vert^{1-q/2}_{W^{1,1}_{\mix}(\itr^{|v|})}\\
  \leq & 2^{\frac{q-1}{2}|v|}\Vert f_v \Vert^{q/2}_{W^{1,q}_{\mix}(\real^s,\varphi)} \Big(\Vert f_v \Vert_{W^{1,q}_{\mix}(\real^s,\varphi)}  \prod_{j\in v}C_{\varepsilon,p,q}\theta^{-1/q-\varepsilon}_j\Big)^{1-q/2}\\
  =&\Vert f_v \Vert_{W^{1,q}_{\mix}(\real^s,\varphi)} \prod_{j\in v}2^{\frac{q-1}{2}}C^{1-q/2}_{\varepsilon,p,q}\theta^{-1/q+1/2-(1-q/2)\varepsilon}_j.
\end{align*}
Putting the above bounds into equation~\eqref{eqn:kvsum}, we get
$$\sum_{\bsk_v\in L_{\bsell,v}}|\hat{f}^w_v(\bsk_v)|^2\leq C_{3}^{|v|} b^{-\alpha\sum_{j\in v}\ell_j}\Vert f_v \Vert^2_{W^{1,q}_{\mix}(\real^s,\varphi)}\prod_{j\in v}\theta^{-\nu_{\alpha,q}-\varepsilon}_j,$$
where $C_3$ depends on $\varepsilon,p,q,b$ and $\alpha$. Finally, \cite[Lemma 3.2]{pan:2025} implies $|\hat{w}_j(k_j)|\leq 4\min(\theta_j,b^{-\ell_j+1})$ and
$$\sum_{k_j=b^{\ell_j-1}}^{b^{\ell_j}-1}|\hat{w}_j(k_j)|^2\leq 16(b-1)\min(\theta^2_jb^{\ell_j-1},b^{-\ell_j+1})\leq 16(b-1) \theta^{1-\alpha}_j b^{-\alpha(\ell_j-1)}.$$
We therefore conclude that
\begin{align*}
   \sigma^2_{\bsell}\leq &2^{|\supp(\bsell)|}b^{-\alpha \Vert\bsell\Vert_1} \sum_{v\subseteq\supp(\bsell)} \Big( C_{3}^{|v|} \Vert f_v \Vert^2_{W^{1,q}_{\mix}(\real^s,\varphi)}\prod_{j\in v}\theta^{-\nu_{\alpha,q}-\varepsilon}_j\Big)\prod_{j\in \supp(\bsell)\setminus v}16(b-1) \theta^{1-\alpha}_jb^\alpha \\
 \leq &  C_{\varepsilon,p,q,b,\alpha}^{|\supp(\bsell)|}b^{-\alpha \Vert\bsell\Vert_1} \sum_{v\subseteq\supp(\bsell)} \Vert f_v \Vert^2_{W^{1,q}_{\mix}(\real^s,\varphi)}\prod_{j\in v}\theta^{-\nu_{\alpha,q}-\varepsilon}_j\prod_{j\in \supp(\bsell)\setminus v}\theta^{1-\alpha}_j 
    \end{align*}
    for $C_{\varepsilon,p,q,b,\alpha}=2\max(C_3,16b^\alpha(b-1))$.
\end{proof}

\subsection{Step III: Bound on scrambled nets variance}
    \begin{proof}[Proof of Theorem~\ref{thm:qmcvar}]
        By Lemma~\ref{lem:qmcvar} and Theorem~\ref{thm:sigmabound},
        \begin{align*}
            &\e \Big|\frac{1}{n}\sum_{i=0}^{n-1}f^w(\bsu_i)-\int_{\itr^s} f^w(\bsu)\rd \bsu\Big|^2 =\frac{1}{n}\sum_{\emptyset\neq \omega \subseteq 1{:}s}\sum_{\bsell\in \natu^\omega} \Gamma_{\omega,\bsell} \sigma^2_{\bsell} \\
            \leq & \frac{1}{b^m} \sum_{\emptyset\neq \omega \subseteq 1{:}s}\sum_{\bsell\in \natu^\omega}\Gamma_{\omega,\bsell} C_{\varepsilon,p,q,b}^{|\omega|}b^{-\alpha \Vert\bsell\Vert_1} \sum_{v\subseteq\omega} \Vert f_v \Vert^2_{W^{1,q}_{\mix}(\real^s,\varphi)}\prod_{j\in v}\theta^{-\nu_{\alpha,q}-\varepsilon}_j\prod_{j\in \omega\setminus v}\theta^{1-\alpha}_j \\
            =&\frac{\Vert f\Vert^2_q}{b^m} \sum_{\emptyset\neq \omega \subseteq 1{:}s}C_{\varepsilon,p,q,b}^{|\omega|}\Big(\sum_{\bsell\in \natu^\omega} \Gamma_{\omega,\bsell}b^{-\alpha\Vert \bsell\Vert_1}\Big)\sum_{v\subseteq \omega}\gamma^2_v\prod_{j\in v} \theta^{-\nu_{\alpha,q}-\varepsilon}_j\prod_{j\in \omega\setminus v} \theta^{1-\alpha}_j.
        \end{align*}
        It is shown in the proof of \cite[Theorem 5.3]{pan:2025} that
        \begin{align*}
          \sum_{\bsell\in \natu^\omega} \Gamma_{\omega,\bsell}b^{-\alpha\Vert \bsell\Vert_1}\leq \frac{C^{|\omega|}_{4}m^{|\omega|-1}b^{(1+\alpha)t_\omega}}{b^{\alpha m}}
        \end{align*}
        for $C_{4}$ depending on $b$ and $\alpha$. Therefore,
        \begin{align*}
            &\e \Big|\frac{1}{n}\sum_{i=0}^{n-1}f^w(\bsu_i)-\int_{\itr^s} f^w(\bsu)\rd \bsu\Big|^2  \\
            \leq 
            &\frac{\Vert f\Vert^2_q}{b^{(1+\alpha)m}}\sum_{\emptyset\neq \omega \subseteq 1{:}s}C_{*}^{|\omega|}m^{|\omega|-1}b^{(1+\alpha)t_\omega}\sum_{v\subseteq \omega}\gamma^2_v\prod_{j\in v} \theta^{-\nu_{\alpha,q}-\varepsilon}_j\prod_{j\in \omega\setminus v} \theta^{1-\alpha}_j.
        \end{align*}
        for $C_{*}=C_{\varepsilon,p,q,b}C_{4}$. The conclusion follows after noticing $\int_{\itr^s} f^w(\bsu)\rd \bsu=\mu$ and $n=b^m$.
    \end{proof}

\section{Application to parameterized PDEs}\label{sec:pdes}
%Consider the model inverse problem of determining the distribution of the random diffusion coefficient of a divergence form elliptic PDEs from observations of a finite set of noisy continuous functionals of the solution. 

\subsection{Model setting}
Consider an elliptic PDEs with a random diffusion coefficient given by
\begin{equation}\label{eq:EPDE}
    \begin{aligned}
    -\nabla_{\bsy}\cdot (a(\bsy,\omega)\nabla_{\bsy} p(\bsy,\omega)) &= g(\bsy)\ \text{in }D,\\
    p(\bsy,\omega) &= 0 \text{ on }\partial D,
\end{aligned}
\end{equation}
with almost every $\omega\in\Omega$ in a complete probability space $(\Omega,\mathcal{F},\mathbb{P})$,
where $D\subseteq \mathbb{R}^d$ $(d=1,2,3)$ is a bounded Lipschitz domain, $\partial D$ is its boundary, and $a(\bsy,\omega)$ is the diffusion coefficient. Denote $H_0^1(D)$ as the subspace of $H^1(D)$ with zero trace on $\partial D$ and the norm $\newnorm{v}_{H_0^1(D)}=\newnorm{\nabla v}_{L^2(D)}$. Let $H^{-1}(D)$ be the dual space of $H_0^1(D)$. The  duality pairing between $H_0^1(D)$ and $H^{-1}(D)$ is denoted by $\inner{\cdot}{\cdot}$. For the numerical approximation of the  problem \eqref{eq:EPDE}, it  is common to represent the input random field $a(\bsy,\omega)$ as a finite sum parameterized by $\bsx\in \real^s$:
\begin{equation}\label{eqn:asyx}
   a(\bsy,\omega)=a_s(\bsy,\bsx) = a_*(\bsy) + a_0(\bsy)\exp\left\{\sum_{j=1}^s x_j \xi_j(\bsy)\right\},
\end{equation}
where $a_*(\bsy)\ge 0$, $a_0(\bsy)>0$ for all $\bsy\in D$ and $\{\xi_j(\bsy)\}_{j=1}^s$ are basis functions.

Following \cite{guth:2025}, we model the parametric inputs $x_j$ as generalized $\beta$-Gaussian random variables with density
\begin{equation}\label{eq:betadensity}
    \varphi(x) = \frac{1}{2\beta^{1/\beta}\Gamma(1+1/\beta)}e^{-\frac{|x|^\beta}{\beta}}\text{ with }\beta>0.
\end{equation}
For $g\in H^{-1}(D)$, we study the problem \eqref{eq:EPDE} with $a(\bsy,\omega)=a_s(\bsy,\bsx)$ in its weak form: seeking a solution $p(\cdot,\bsx)\in H_0^1(D)$ such that 
$$\int_D a_s(\bsy,\bsx)\nabla_x p(\bsy,\bsx)\cdot \nabla_{\bsy} v(\bsy)\mrd \bsy=\inner{g}{v}$$
for all $v\in H_0^1(D)$ and almost $\bsx\in \real^s$.

Given $G\in H^{-1}(D)$ independent of $\bsx$, our goal is to estimate the expectation of $f(\bsx) = \inner{G}{p(\cdot,\bsx)}$ and investigate the convergence rate of our proposed method. 
More generally, we assume that for any $v\subseteq 1{:}s$, 
\begin{equation}\label{eq:normpartialp}
  \newnorm{\partial^v p(\cdot,\bsx)}_{H_0^1(D)}\lesssim (|v|!)^\sigma\left(\prod_{j\in v} b_j\right)\exp\left\lbrace \sum_{j=1}^s \alpha_j|x_j|^\tau\right\rbrace,  
\end{equation}
with $\sigma\ge 1$, $0<\tau\le \beta$, $\bsb =(b_j)_{j\geq 1}\subseteq \natu_0$, and $\bsalpha = (\alpha_j)_{j\geq 1}\subseteq \natu$. The notation $\lesssim$ from now on is used for hiding constant independent of $s$ and $\bsx$. We assume that  $\newnorm{\bsalpha}_1=\sum_{j=1}^\infty\alpha_j<\infty$ so that $\sum_{j=1}^\infty \alpha_j|x_j|^\tau<\infty$ almost surely, which holds for Gevrey regular input random fields \cite{guth:2025}. When $\beta=\tau$, we additionally assume $\beta<1/\newnorm{\bsalpha}_{\infty}$ so that $\exp(\newnorm{\bsalpha}_{\infty} |x|^\tau)\in L^1(\real,\varphi)$ (see \eqref{eqn:Ialpha}).

As a special case, for lognormal random fields given by \eqref{eqn:asyx} with all $x_j$  independently following a standard Gaussian distribution, we have  $\beta=2$, $\sigma=\tau=1$, $\alpha_j=\newnorm{\xi_j(\bsy)}_{L^\infty(\Bar{D})}$ and $b_j=\alpha_j/\ln 2$ (see \cite{graham:2015} for details). 

\subsection{Bounds on $\gamma_v$}
To bound $\gamma_v$, we assume the probability density $\varphi(x)$ satisfies the generalized Poincaré inequality as follows:

\begin{assumption}\label{assump:Poincare}
Given $q>1$, assume that there exists a constant $PC_q>0$ such that for every $f\in W^{1,q}(\real,\varphi)$ and $\mu (f)=\int_\real f(x)\varphi(x)\rd x$,
\begin{equation}\label{eqn:Poincare}
 \int_\real |f(x)-\mu (f)|^q \varphi(x)\rd x\leq PC_q \int_\real |f'(x)|^q \varphi(x)\rd x.   
\end{equation}
\end{assumption}

For $q=2$, \cite{bobkov1999isoperimetric} establishes that Assumption~\ref{assump:Poincare} holds when $\varphi(x)=\exp(-V(x))$ for a convex function $V:\real\to (-\infty,\infty]$. For general $q>1$, the following lemma shows that Assumption~\ref{assump:Poincare} remains valid for several common density functions, including generalized $\beta$-Gaussian with $\beta\ge 1$. 

\begin{lemma}\label{lem:Poincareinequality}
Let $m(x)=(1-\Phi(x))/\varphi(x)$ be the Mills ratio. If $\varphi$ satisfies Assumption~\ref{assump:rho} and $\sup_{x\in [0,\infty)}m(x)<\infty$, then Assumption~\ref{assump:Poincare} holds for every $q>1$.
\end{lemma}
\begin{proof}
    See Appendix~\ref{appB}.
\end{proof}

\begin{remark}
    For the standard Gaussian distribution, it is known that its Mills ratio satisfies $m(x)< x^{-1}$ for $x>0$. An analogous bound, $m(x)\le x^{1-\beta}$, holds for the generalized $\beta$-Gaussian distribution with $\beta\ge 1$.  This is because for any $\beta\ge 1$ and $x>0$
    \begin{align*}
        1-\Phi(x) &=\frac{1}{2\beta^{1/\beta}\Gamma(1+1/\beta)} \int_x^\infty e^{-\frac{t^\beta}{\beta}}\rd t\\
        &\le  \frac{1}{2\beta^{1/\beta}\Gamma(1+1/\beta)} \int_x^\infty \left(\frac{t}{x}\right)^{\beta-1}e^{-\frac{t^\beta}{\beta}}\rd t\\
        &=\frac{x^{1-\beta}}{2\beta^{1/\beta}\Gamma(1+1/\beta)} e^{-\frac{x^\beta}{\beta}}\\
        &=x^{1-\beta}\varphi(x).
    \end{align*}
Lemma~\ref{lem:Poincareinequality} is therefore applicable whenever $\beta\geq 1$.
\end{remark}

To ensure the interchange of integration and differentiability, we make the following assumption.
\begin{assumption}\label{assump:interchange}
Assume that for $f\in W^{1,q}_{\mix}(\real^s,\varphi)$ and $v,w\subseteq 1{:}s$ with $v\cap w=\emptyset$,
$$
\int_{\real^{|w|}}\partial^v f(\bsx) \prod_{j\in w}\varphi(x_j)\rd \bsx_{w}=\partial^v\int_{\real^{|w|}} f(\bsx) \prod_{j\in w}\varphi(x_j)\rd \bsx_{w}
$$  
for almost every $\bsx_{w^c}\in \real^{s-|w|}$ (with respect to the Lebesgue measure).
\end{assumption}

\begin{lemma}\label{lem:interchange}
Assumption~\ref{assump:interchange} holds if $\varphi$  is continuous  and strictly positive over $\real$.
\end{lemma}
\begin{proof}
    See Appendix~\ref{appC}.
\end{proof}

\begin{lemma}\label{lem:Poincare}
    Suppose that Assumptions~\ref{assump:Poincare} and \ref{assump:interchange} hold with $q>1$. Then for $f\in W^{1,q}_{\mix}(\real^s,\varphi)$ with ANOVA terms $\{f_v\mid v\subseteq 1{:}s\}$,
    $$\Vert f_v \Vert_{W^{1,q}_{\mix}(\real^s,\varphi)}\leq (PC_q+1)^{|v|/q} \Vert \partial^v f_v \Vert_{L^q(\real^s,\varphi)} \ \forall v\subseteq 1{:}s. $$
\end{lemma}
\begin{proof}
The proof is inspired by \cite{effdimsobononper}.   For $\omega\subsetneq v$ and $j\in v\setminus \omega$, the definition of ANOVA decomposition and Assumption~\ref{assump:interchange} imply
    $$\int_{\real} \partial^w f_v(\bsx)\varphi(x_j)\rd x_j= \partial^w\Big(\int_{\real} f_v(\bsx)\varphi(x_j)\rd x_j\Big)=0.$$
    Denoting $\omega'=\omega\cup \{j\}$, we can apply Assumption~\ref{assump:Poincare} with respect to $x_j$ and bound
    $$\int_\real |\partial^\omega f_v(\bsx)|^q \varphi(x_j)\rd x_j\leq PC_q \int_\real |\partial^{\omega'}f(\bsx)|^q \varphi(x_j)\rd x_j.$$
    An induction shows that
    $$\int_{\real^{|v\setminus\omega|}} |\partial^\omega f_v(\bsx)|^q \prod_{j\in v\setminus \omega} \varphi(x_j)\rd \bsx_{v\setminus \omega}\leq PC^{|v|-|\omega|}_q \int_{\real^{|v\setminus\omega|}} |\partial^{v}f_v(\bsx)|^q \prod_{j\in v\setminus \omega} \varphi(x_j)\rd \bsx_{v\setminus \omega}.$$
   Further integrating both sides with respect to $\prod_{j\in (v\setminus\omega)^c} \varphi(x_j)\rd \bsx_{(v\setminus\omega)^c}$, we obtain
    $$\Vert \partial^\omega  f_v \Vert^q_{L^q(\real^s,\varphi)}\leq PC^{|v|-|\omega|}_q \Vert \partial^v f_v \Vert^q_{L^q(\real^s,\varphi)} . $$
    Finally, because $f_v$ does not depend on $\bsx_{v^c}$,
    \begin{align*}
     \Vert f_v \Vert^q_{W^{1,q}_{\mix}(\real^s,\varphi)}= \sum_{\omega\subseteq v} \Vert \partial^\omega  f_v \Vert^q_{L^q(\real^s,\varphi)}\leq &\sum_{\omega\subseteq v} PC^{|v|-|\omega|}_q\Vert \partial^v f_v \Vert^q_{L^q(\real^s,\varphi)}\\
     =&(PC_q+1)^{|v|}\Vert \partial^v f_v \Vert^q_{L^q(\real^s,\varphi)}.   
    \end{align*}
\end{proof}

\begin{theorem}\label{thm:PODweight}
  Let  $f(\bsx) = \inner{G}{p(\cdot,\bsx)}$ for $G\in H^{-1}(D)$. Assume that $\varphi(x)$ is the density of generalized $\beta$-Gaussian given by \eqref{eq:betadensity}, and \eqref{eq:normpartialp} holds with $\sigma\ge 1$, $0<\tau\le \beta$, and $\newnorm{\bsalpha}_1<\infty$. Then $f\in W^{1,q}_{\mix}(\real^s,\varphi)$ for all $q>1$ when $\beta>\tau$, or for $q\in (1,1/(\beta\newnorm{\bsalpha}_{\infty}))$ when $\beta=\tau<1/\newnorm{\bsalpha}_{\infty}$. Furthermore, for nonzero $f$ and all $v\subseteq 1{:}s$,
  \begin{equation}\label{eq:boundgammav}
      \gamma_v=\Vert f\Vert^{-1}_{q} \Vert f_v \Vert_{W^{1,q}_{\mix}(\real^s,\varphi)}\lesssim (|v|!)^{\sigma}\prod_{j\in v} \Gamma_j,
  \end{equation}
  where $\Gamma_j=(PC_q+1)^{1/q}b_j>0$.
\end{theorem}
\begin{proof}
The linearity of $G$ together with  \eqref{eq:normpartialp} gives
\begin{align}\label{eq:partialuf}
|\partial^v f(\bsx)|&\le  \newnorm{G}_{H^{-1}(D)}\newnorm{\partial^v p(\cdot,\bsx)}_{H_0^1(D)} \\
&\lesssim (|v|!)^\sigma\left(\prod_{j\in v} b_j\right)\exp\left\lbrace \sum_{j=1}^s \alpha_j|x_j|^\tau\right\rbrace.\notag
\end{align}
Consequently,
\begin{align}\label{eq:partialvfv}
&\Vert \partial^v f_v \Vert_{L^q(\real^s,\varphi)}=\left(\int_{\real^{|v|}}\Big|\int_{\real^{s-|v|}}\partial^v f(\bsx) \prod_{j\in v^c}\varphi(x_j)\rd \bsx_{v^c} \Big|^q \prod_{j\in v}\varphi(x_j)\rd \bsx_v\right)^{1/q} \\
\lesssim & (|v|!)^{\sigma}\left(\prod_{j\in v} b_j\right)\left(\int_{\real^{s-|v|}}\prod_{j\in v^c} e^{\alpha_j|x_j|^\tau}\varphi(x_j)\rd \bsx_{v^c}\right)\left(\int_{\real^{|v|}} \prod_{j\in v} e^{q\alpha_j|x_j|^\tau}\varphi(x_j)\rd \bsx_v\right)^{1/q}\notag\\
\lesssim&(|v|!)^{\sigma}\left(\prod_{j\in v} b_j\right)\left(\prod_{j\in v^c}I(\alpha_j)\right)\left(\prod_{j\in v} I(q\alpha_j)\right)^{1/q},\notag
\end{align}
where 
$$
I(\alpha) = \int_{\real} e^{\alpha|x|^\tau}\varphi(x)\mrd x=\frac{1}{2\beta^{1/\beta}\Gamma(1+1/\beta)}\int_{\real} e^{\alpha|x|^\tau-\beta^{-1}|x|^\beta}\mrd x.
$$

If $\beta = \tau$, to ensure $I(q\alpha_j)<\infty$ for all $j$, we require that $q\alpha_j<1/\beta$ for all $j$, i.e., $q<1/(\beta\newnorm{\bsalpha}_{\infty})$. Then 
\begin{equation}\label{eqn:Ialpha}
  I(\alpha) = \frac{1}{\beta^{1/\beta}\Gamma(1+1/\beta)}\int_0^{\infty} e^{(\alpha-\beta^{-1})x^\beta}\mrd x=\frac{1}{(1-\alpha\beta)^{1/\beta}}.  
\end{equation}
Since $(1-x)^{-1}\le \exp(x(1-x)^{-1})$ for all $x\in [0,1)$, for any $w\subseteq 1{:}s$,
$$
\prod_{j\in w}I(\alpha_j)\le \prod_{j=1}^s \frac{1}{(1-\alpha_j\beta)^{1/\beta}}\le \exp\left(\sum_{j=1}^s \frac{\alpha_j}{1-\alpha_j\beta}\right)\le\exp\left(\frac{\newnorm{\bsalpha}_{1}}{1-\beta\newnorm{\bsalpha}_{\infty}}\right)<\infty.
$$
By Lemma~\ref{lem:interchange}, Assumption~\ref{assump:interchange} holds for generalized $\beta$-Gaussian densities.
It then follows Lemma~\ref{lem:Poincare} and \eqref{eq:partialvfv} that
\begin{align*}
  \Vert f_v \Vert_{W^{1,q}_{\mix}(\real^s,\varphi)}&\leq (PC_q+1)^{|v|/q} \Vert \partial^v f_v \Vert_{L^q(\real^s,\varphi)} \\
  &\lesssim C_q^{|v|} (|v|!)^{\sigma}\left(\prod_{j\in v} b_j\right)\left(\prod_{j\in v^c}I(\alpha_j)\right)\left(\prod_{j\in v} I(q\alpha_j)\right)^{1/q}\\
  &\lesssim C_q^{|v|} (|v|!)^{\sigma}\left(\prod_{j\in v} b_j\right)\exp\left(\frac{\newnorm{\bsalpha}_{1}}{1-\beta\newnorm{\bsalpha}_{\infty}}+\frac{\newnorm{\bsalpha}_{1}}{1-q\beta \newnorm{\bsalpha}_{\infty}}\right)\\
  &\lesssim (|v|!)^{\sigma}\prod_{j\in v} (C_q b_j),
\end{align*}
where $C_q = (PC_q+1)^{1/q}>0$.

If $\beta>\tau$, by the proof of \cite[Proposition 4.2]{guth:2025}, for any $w\subseteq 1{:}s$,
$$
\prod_{j\in w}I(\alpha_j)\lesssim \exp\left(2\newnorm{\bsalpha}_1\right)<\infty.
$$
We similarly have
\begin{align*}
  \Vert f_v \Vert_{W^{1,q}_{\mix}(\real^s,\varphi)}&\lesssim C_q^{|v|} (|v|!)^{\sigma}\left(\prod_{j\in v} b_j\right)\exp\left(4\newnorm{\bsalpha}_1\right)\lesssim (|v|!)^{\sigma}\prod_{j\in v} (C_q b_j).
\end{align*}
We complete the proof by noticing that $\gamma
_v=\Vert f_v \Vert_{W^{1,q}_{\mix}(\real^s,\varphi)}/\newnorm{f}_q\lesssim \Vert f_v \Vert_{W^{1,q}_{\mix}(\real^s,\varphi)}\lesssim (|v|!)^{\sigma}\prod_{j\in v} \Gamma_j$ for $\Gamma_j=C_qb_j$.
\end{proof}

Theorem~\ref{thm:PODweight} yields product and order dependent (POD) weights $(|v|!)^{\sigma}\prod_{j\in v} \Gamma_j$ for $\gamma_v$ rather than the product weights  $\prod_{j\in v} \Gamma_j$ appearing in Corollary~\ref{cor:qmcvar}.  To reduce the problem to the product weight case, we use the bound
\begin{equation}\label{eqn:PODtoproduct}
\gamma_v\lesssim (|v|!)^{\sigma}\prod_{j\in v} \Gamma_j\le \prod_{j\in v} (j^{\sigma} \Gamma_j)=:\prod_{j\in v} \Gamma'_j,    
\end{equation}
and then apply Corollary~\ref{cor:qmcvar} with $\Gamma_j$ replaced by $\Gamma'_j=j^{\sigma} \Gamma_j$. 

Suppose that $b_j=O(j^{-\rho^*})$ for some $\rho^*>0$.
Then $\gamma_v\lesssim \prod_{j\in v}\Gamma'_j$ with $\Gamma'_j = O(j^{-\rho})$ and $\rho=\rho^*-\sigma$. If $\rho>\max(2/q,1)$, as discussed in Remark~\ref{rmk:optimalalpha},  the proposed method attains a dimension-independent MSE rate of nearly $O(n^{-1-\alpha^*(q,\rho)})$, where  $\alpha^*(q,\rho)$ is given by \eqref{eq:alphastar}. Following Remark~\ref{rem:qinfinity} and taking $q$ as large as permitted by Theorem~\ref{thm:PODweight}, we obtain the dimension-independent convergence rates stated in the following corollary.
%\red{This relaxation is tight in the sense that
% $\sum_{v\subseteq \natu }  \gamma'_vm^{|v|}$ is sub-exponential in $m$ if and only if $\rho-\rho'>1$ (see \cite[Theorem 3]{pan2025sharpconvergenceboundssums} for a rigorous statement).} 

\begin{corollary}\label{cor:pod}
    Consider the settings in Theorem~\ref{thm:PODweight} and assume that $b_j=O(j^{-\rho^*})$ with $\rho=\rho^*-\sigma>1$. If $\beta=\tau<1/\newnorm{\bsalpha}_{\infty}$, we additionally require $\rho>2\beta\newnorm{\bsalpha}_{\infty}$.  By setting $q^*=1/\beta\newnorm{\bsalpha}_{\infty}$ and taking
    \begin{equation*}
    \theta_j=
    \begin{cases}
    \theta_0 j^{-\rho q^*},&\text{ if } \beta=\tau<1/\newnorm{\bsalpha}_{\infty}, \ q^*\in (1,2]\\
     \theta_0 j^{-2\rho/(2\alpha^* /q^*+1-\alpha^* )}, &\text{ if } \beta=\tau<1/\newnorm{\bsalpha}_{\infty}, \ q^*>2\\
        \theta_0 j^{-2\rho/(3-2\rho)}, &\text{ if }\beta>\tau, \ \rho\in (1,3/2)\\
    0, &\text{ if } \beta>\tau,\ \rho\ge 3/2,
    \end{cases}
    \end{equation*}
    with $\theta_0\in(0,1/2)$,
    we have
    \begin{equation*}
       \e \Big|\frac{1}{n}\sum_{i=0}^{n-1}f^w(\bsu_i)-\int_{\real^s} f(\bsx) \prod_{j=1}^s \varphi(x_j)\rd \bsx\Big|^2 
           \le C n^{-1-\alpha^*+\varepsilon}
    \end{equation*}
    with
    \begin{align*}
      \alpha^* =
\begin{cases}
  \frac{\rho q^* -2}{\rho q^*+1}, &\text{ if } \beta=\tau<1/\newnorm{\bsalpha}_{\infty}, \ q^*\in (1,2]\\
   1-\frac{\sqrt{((2\rho-3)q^*+8)^2+24(q^*-2)}-(2\rho-3)q^*-8}{2(q^*-2)}, &\text{ if } \beta=\tau<1/\newnorm{\bsalpha}_{\infty}, \ q^*>2\\
   2\rho-2, &\text{ if } \beta>\tau,\ \rho\in (1,3/2)\\
   1, &\text{ if } \beta>\tau,\ \rho\ge 3/2,\\
\end{cases}      
    \end{align*}
where $C>0$ is a constant independently of the dimension $s$ and $\varepsilon>0$ can be arbitrarily small. 
\end{corollary}

\begin{remark}
    When $\beta=\tau$, the BDIS method attains an MSE rate arbitrarily close to $O(n^{-2})$  as $\rho\to\infty$. This is in contrast with the inversion method, whose MSE rate cannot converge faster than $O(n^{-2+2\beta\newnorm{\bsalpha}_{\infty}})$ even when $\rho\to\infty$ \cite[Theorem 6.2]{guth:2025}.
\end{remark}

\section{Connection with weighted Sobolev norms}\label{sec:weightedSobolev}

When using lattice rules, a common way to characterize $f$ is through the following weighted Sobolev norm \cite{graham:2015}:
$$\Vert f \Vert_{s,\bsgamma'}=\left[\sum_{v\subseteq 1{:}s}\frac1{\gamma'_v} \int_{\real^{|v|}}\Big(\int_{\real^{s-|v|}}\partial^v f(\bsx) \prod_{j\in v^c}\varphi(x_j)\rd \bsx_{v^c} \Big)^2 \prod_{j\in v}\psi_j^2(x_j)\rd \bsx_v\right]^{1/2},$$
where $\bsgamma'=\{\gamma'_v\mid v\subseteq 1{:}s\}$ is a set of positive weights and $\psi_j(x)$ is a coordinate-dependent weight function. Under appropriate choice of $\bsgamma'$ and $\psi(x)$, \cite{graham:2015} proves $\Vert f \Vert_{s,\bsgamma'}$ can be bounded independent of $s$ and efficient lattice-based QMC  rules can be constructed to estimate $\mathbb{E}[f(\bsx)]$.

Below we show $\Vert f \Vert_{s,\bsgamma'}$ naturally upper bounds $\Vert f\Vert_2$ defined as in Theorem~\ref{thm:qmcvar}, thereby bridging the framework in \cite{graham:2015} with ours. For simplicity, we restrict our discussion to the case $\psi_j^2(x)=\varphi(x)$. When $x_j$ is normally distributed, the Gaussian tails of this $\psi_j(x)$ yield a weaker norm than the exponential-tailed alternative in \cite{graham:2015}, thereby accommodating a broader class of functions.

To state the next theorem, we generalize $\Vert f \Vert_{s,\bsgamma'}$ by drawing an analogy to the $\Vert \cdot \Vert_{s,\alpha,\bsgamma,q,r}$-norm introduced in \cite{dick2014higher,dick2016multilevel} and defining
$$\Vert f \Vert_{s,\bsgamma',q,r}=\left[\sum_{v\subseteq 1{:}s}\Bigg(\frac{1}{\gamma'_v}\int_{\real^{|v|}}\Bigg|\int_{\real^{s-|v|}}\partial^v f(\bsx) \prod_{j\in v^c}\varphi(x_j)\rd \bsx_{v^c} \Bigg|^q \prod_{j\in v}\varphi(x_j)\rd \bsx_v\Bigg)^{r/q}\right]^{1/r}$$
for $r\in [1,\infty)$, and 
$$\Vert f \Vert_{s,\bsgamma',q,\infty}=\sup_{v\subseteq 1{:}s}\Bigg(\frac{1}{\gamma'_v}\int_{\real^{|v|}}\Bigg|\int_{\real^{s-|v|}}\partial^v f(\bsx) \prod_{j\in v^c}\varphi(x_j)\rd \bsx_{v^c} \Bigg|^q \prod_{j\in v}\varphi(x_j)\rd \bsx_v\Bigg)^{1/q}.$$

\begin{theorem}\label{thm:weightednorm}
    Suppose Assumptions~\ref{assump:Poincare} and \ref{assump:interchange} hold with $q>1$. Then for $r\in [1,\infty]$ and $f\in W^{1,q}_{\mix}(\real^s,\varphi)$ with $\Vert f \Vert_{s,\bsgamma',q,r}<\infty$,
    \begin{equation}\label{eqn:leqweightednorm}
     \Vert f\Vert_q\leq \begin{dcases}
    \left[\sum_{v\subseteq 1{:}s}\Bigg(\frac{\gamma^q_v}{(PC_q+1)^{|v|}\gamma'_v}\Bigg)^{r/q}\right]^{-1/r}\Vert f \Vert_{s,\bsgamma',q,r}, &\text{ if } r\in [1,\infty) \\
   \sup_{v\subseteq 1{:}s}\left(\frac{\gamma^q_v}{(PC_q+1)^{|v|}\gamma'_v}\right)^{-1/q}   \Vert f \Vert_{s,\bsgamma',q,\infty} , &\text{ if } r=\infty
\end{dcases}
     , 
    \end{equation}
    and
    \begin{equation}\label{eqn:gammavfq}
      \gamma_v\Vert f\Vert_q\leq \left((PC_q+1)^{|v|}\gamma'_v\right)^{1/q}\Vert f \Vert_{s,\bsgamma',q,r}, 
    \end{equation}
    where $\Vert f\Vert_q$ and $\gamma_v$ are defined in \eqref{eq:fq} and \eqref{eq:gammav}, respectively.
\end{theorem}
\begin{proof}
We prove the result for $r\in[1,\infty)$, as the  $r=\infty$ case follows from a similar argument.
    By the definition of ANOVA decomposition and Assumption~\ref{assump:interchange}, for every $v\subseteq 1{:}s$,
    $$\int_{\real^{s-|v|}}\partial^v f(\bsx) \prod_{j\in v^c}\varphi(x_j)\rd \bsx_{v^c}=\partial^v\Big(\sum_{\omega\subseteq v} f_\omega (\bsx)\Big)=\partial^vf_v(\bsx).$$
    Therefore, Lemma~\ref{lem:Poincare} implies
    $$\Vert f \Vert^r_{s,\bsgamma',q,r}=\sum_{v\subseteq 1{:}s} \Big(\frac{1}{\gamma'_v} \Vert \partial^v f_v \Vert^q_{L^q(\real^s,\varphi)}\Big)^{r/q}\geq \sum_{v\subseteq 1{:}s} \Big(\frac{1}{\gamma'_v (PC_q+1)^{|v|}} \Vert f_v \Vert^q_{W^{1,q}_{\mix}(\real^s,\varphi)} \Big)^{r/q}.$$
      Plugging in $ \Vert f_v \Vert_{W^{1,q}_{\mix}(\real^s,\varphi)}=\gamma_v \Vert f\Vert_q$ yields equation~\eqref{eqn:leqweightednorm}.  Bounding the sum over $v\subseteq 1{:}s$ from below by one of its terms yields Equation~\eqref{eqn:gammavfq}. 
\end{proof}

% \begin{theorem}\label{thm:weightednorm}
%     Suppose $\varphi$ satisfies Assumption~\ref{assump:Poincare} with $q>1$. Then for $f\in W^{1,q}_{\mix}(\real^s,\varphi)$ with $\Vert f \Vert_{s,\bsgamma',q}<\infty$,
%     \begin{equation}\label{eqn:gammavfq}
%      \gamma_v\Vert f\Vert_q\leq \left((PC_q+1)^{|v|}\gamma'_v\right)^{1/q}\Vert f \Vert_{s,\bsgamma',q} 
%     \end{equation}
%     and 
%     $$ \Vert f\Vert_q\leq \Big(\sup_{v\subseteq 1{:}s}(PC_q+1)^{|v|}\gamma'_v\Big)^{1/q}\Vert f \Vert_{s,\bsgamma',q},$$
%     where $\gamma_v$ and $\Vert f\Vert_q$ are defined as in Theorem~\ref{thm:qmcvar}.
% \end{theorem}
% \begin{proof}
%     By the definition of ANOVA decomposition,
%     $$\int_{\real^{s-|v|}}\partial^v f(\bsx) \prod_{j\in v^c}\varphi(x_j)\rd \bsx_{v^c}=\partial^v\Big(\sum_{\omega\subseteq v} f_\omega (\bsx)\Big)=\partial^vf_v(\bsx).$$
%     Therefore, for every $v\subseteq 1{:}s$, Lemma~\ref{lem:Poincare} implies
%     $$\Vert f \Vert^q_{s,\bsgamma',q}\geq \frac{1}{\gamma'_v} \Vert \partial^v f_v \Vert^q_{L^q(\real^s,\varphi)}\geq \frac{1}{\gamma'_v (PC_q+1)^{|v|}}  \Vert f_v \Vert^q_{W^{1,q}_{\mix}(\real^s,\varphi)}.$$
%     The first conclusion follows after plugging in $ \Vert f_v \Vert_{W^{1,q}_{\mix}(\real^s,\varphi)}=\gamma_v \Vert f\Vert_q$. The second conclusion follows because $\sup_{v\subseteq 1{:}s} \gamma_v=1$.
% \end{proof}

\begin{corollary}\label{cor:weightednorm}
    Under the assumptions of Theorems~\ref{thm:qmcvar} and \ref{thm:weightednorm},
    \begin{equation}\label{eqn:weightednorm}
  \e \Big|\frac{1}{n}\sum_{i=0}^{n-1}f^w(\bsu_i)-\int_{\real^s} f(\bsx) \prod_{j=1}^s \varphi(x_j)\rd \bsx\Big|^2 
            \leq \frac{\Vert f \Vert^2_{s,\bsgamma',q,r}}{n^{1+\alpha}}\sum_{\emptyset \neq \omega\subseteq 1{:}s }  C^{|\omega|}_{**} m^{|\omega|-1}\Tilde{\gamma}'_\omega,    
    \end{equation}
            where $C_{**}$ is a constant depending on $\varepsilon,p,q,b,\alpha$ and $PC_q$, and
            $$\Tilde{\gamma}'_\omega=b^{(1+\alpha)t_\omega} \sum_{v\subseteq \omega}(\gamma'_v)^{2/q}\prod_{j\in v}  \theta^{-\nu_{\alpha,q}-\varepsilon}_j \prod_{j\in \omega\setminus v} \theta^{1-\alpha}_j.$$
\end{corollary}
\begin{proof}
     This is an immediate consequence of \eqref{eqn:maintheorem} and \eqref{eqn:gammavfq}.
\end{proof}

Corollary~\ref{cor:weightednorm} suggests our results in Section~\ref{sec:main} are applicable to $f$ with $\Vert f \Vert_{s,\bsgamma',q,r}<\infty$ by simply replacing $\Vert f\Vert_q$ with $\Vert f \Vert_{s,\bsgamma',q,r}$ and $\gamma_v$ with $(\gamma'_v)^{1/q}$.

\begin{remark}
   When $\gamma'_v$ takes the POD form described in Section~\ref{sec:pdes}, an argument analogous to \eqref{eqn:PODtoproduct} reduces the problem to the product weight case. See \cite{pan2025sharpconvergenceboundssums} for further estimates of the right-hand side of \eqref{eqn:weightednorm} when $\gamma'_v$ takes a non-product form. 
\end{remark}

\section{Numerical experiments}\label{sec:numer}

In this section, we present numerical results to validate the efficiency of the proposed method by  comparing the convergence rates of BDIS-based RQMC against  MC and standard RQMC (using inversion method, i.e., all $\theta_j=0$). We use scrambled Sobol' sequences \cite{rtms} for RQMC methods. For the BDIS methods, we take $p=1$ in \eqref{eq:tildeetap} and compare several choices of $w_j=w_{\theta_j}$.

Now consider the elliptic PDE problem \eqref{eq:EPDE} on the unit square domain $D = (0, 1)^2$ with a deterministic source term given by $g(\bsy) = y_2$ for $\bsy=(y_1,y_2)$. The diffusion coefficient $a(\bsy,\omega)$ follows a truncated random field model:
\begin{equation}
    a(\bsy, \omega) = a_s(\bsy, \bsx)=\exp\left( \sum_{j=1}^{s} x_j \xi_j(\boldsymbol{y}) \right),
\end{equation}
where $x_j$ are independent $\beta$-Gaussian random variables with density given by \eqref{eq:betadensity}, and the basis functions $\{\xi_j\}_{j=1}^s$ share a fixed spatial mode, with amplitudes decaying according to the index $j$:
\begin{equation}
    \xi_j(\bsy) = \zeta \cdot j^{-\rho^*} \sin(\pi y_1) \sin(\pi y_2), 
\end{equation}
where $\zeta>0$ and $j=1, \dots, s$.
The parameter $\rho^* > 1$ controls the decay rate of the fluctuations, thereby determining the decay rate of $\gamma_v$. As shown in \cite{graham:2015}, this model satisfies \eqref{eq:normpartialp} with  $
\tau=\sigma=1$, $\alpha_j=\newnorm{\xi_j(\bsy)}_{L^\infty(\Bar{D})}=\zeta \cdot j^{-\rho^*}$ and $b_j=O(j^{-\rho^*})$. Clearly, we have  $\newnorm{\bsalpha}_{\infty}=\zeta$ and $\newnorm{\bsalpha}_{1}<\infty$ since $\rho^*>1$.

The spatial domain is discretized using the finite element method with continuous piecewise linear ($P_1$) basis functions on a uniform triangular mesh. We set the mesh refinement level to $N_{\text{mesh}}=8$, resulting in a triangulation with $h \approx 1/8$. In the following analysis, we take no account of the dimension truncation error and the spatial discretization error, which are independent of the quadrature error. We refer to \cite{guth:2025,kuo:nuye:2016} for a comprehensive analysis of the total error.  For a given truncation dimension $s$ and a mesh refinement level, we report  root mean squared errors (RMSEs) of different quadrature rules.  We fix the truncation dimension to $s=128$.

For each fixed sample size $n$, we perform $R=64$ independent scrambling for RQMC methods. Let $\hat p_n^{(r)}(\cdot)$ denote the numerical approximation of the solution $p(\cdot,\bsx)$ based on the $r$-th scrambling of $n$ Sobol' points, $r=1, \dots, R$. The overall estimator for the expected value $\mathbb{E}[p(\cdot,\bsx)]$ is given by the sample mean over the $R$ scramblings:
\begin{equation}
    \bar p_{n,R}(\cdot)= \frac{1}{R} \sum_{r=1}^{R} \hat p_n^{(r)}(\cdot).
\end{equation}

The convergence behavior is analyzed by computing the RMSE as the sample size $n$ increases geometrically from $2^6$ to $2^{16}$. We estimate the RMSE using the sample standard deviation over the $R$ independent replications. Specifically, we measure the error in two ways: the $H_0^1(D)$-norm error
\begin{equation*}
    \hat{\sigma}_{1} =\sqrt{\frac{1}{R-1} \sum_{r=1}^{R} \|\bar p_{n,R}- \hat p_n^{(r)}\|_{H_0^1(D)}^2}= \sqrt{\frac{1}{R-1} \sum_{r=1}^{R} \|\nabla\bar p_{n,R}- \nabla \hat p_n^{(r)} \|_{L_2(D)}^2},
\end{equation*}
where the integrals in the $L_2(D)$-norm are computed exactly using the finite element representation of $\hat p_n^{(r)}$, and the pointwise error at the domain center $\bsy_c = (0.5, 0.5)$
\begin{equation*}
    \hat{\sigma}_{2} = \sqrt{\frac{1}{R-1} \sum_{r=1}^{R} \left| \bar p_{n,R}(\bsy_c) - \hat p_n^{(r)}(\bsy_c) \right|^2}.
\end{equation*}

We set the decay parameter $\rho^*=4$, $\newnorm{\bsalpha}_{\infty}=\zeta=2/3$, and the distribution parameter $\beta=1$. Since $\tau=\beta=1$, the integrand belongs to $W^{1,q}_{\mix}(\real^s,\varphi)$ with $1<q<1/(\beta\newnorm{\bsalpha}_{\infty})=3/2$. For this setting, we have $\gamma_v\le  \prod_{j\in v}\Gamma_j$ with $\Gamma_j = O(j^{-\rho})$ and $\rho=\rho^*-\sigma=3$. As suggested by Corollary~\ref{cor:pod}, we set $\theta_j = \theta_0 j^{-4.5}$ to attain a dimension-independent RMSE rate of nearly $O(n^{-8/11})$. We also consider other two cases:  $\theta_j = \theta_0j^{-2}$ and $\theta_{j} = \theta_0j^{-6}$. The former case is an underestimate of the optimal decay rate of $\theta_j$, while the latter case is an overestimate derived by neglecting the factorial term $(|v|!)^\sigma$ (originating from \eqref{eq:normpartialp}) in the upper bound \eqref{eq:boundgammav} of $\gamma_v$, resulting in $\rho=\rho^*=4$. We set $\theta_0=0.1$ in our numerical experiments.

Figure~\ref{fig:convergence_results} reports the results. Due to the rapid boundary growth of the integrand, the standard RQMC method converges slowly, whereas the proposed BDIS method with three choices of $\theta_j$ yields a faster convergence rate than the usual Monte Carlo rate $O(n^{-1/2})$. For the case of $\theta_j=0.1 j^{-4.5}$, the BDIS method yields an RMSE close to $O(n^{-8/11})\approx O(n^{-0.73})$, which is consistent with our theory. Interestingly, the BDIS method with $\theta_j=0.1 j^{-6}$, obtained by omitting the factorial term, performs even better. This omission strategy was previously employed in \cite{kaarnioja:2022b} to define a new class of product weights for lattice-based kernel approximation, leading to improved performance in high dimensions compared to smoothness-driven POD weights (which include the factorial term). These results suggest that the estimate $\sigma=1$ in \eqref{eq:normpartialp} may be overly conservative.

\begin{figure}[ht]
    \centering
    \begin{subfigure}[b]{0.48\textwidth}
        \centering
        \includegraphics[width=\textwidth]{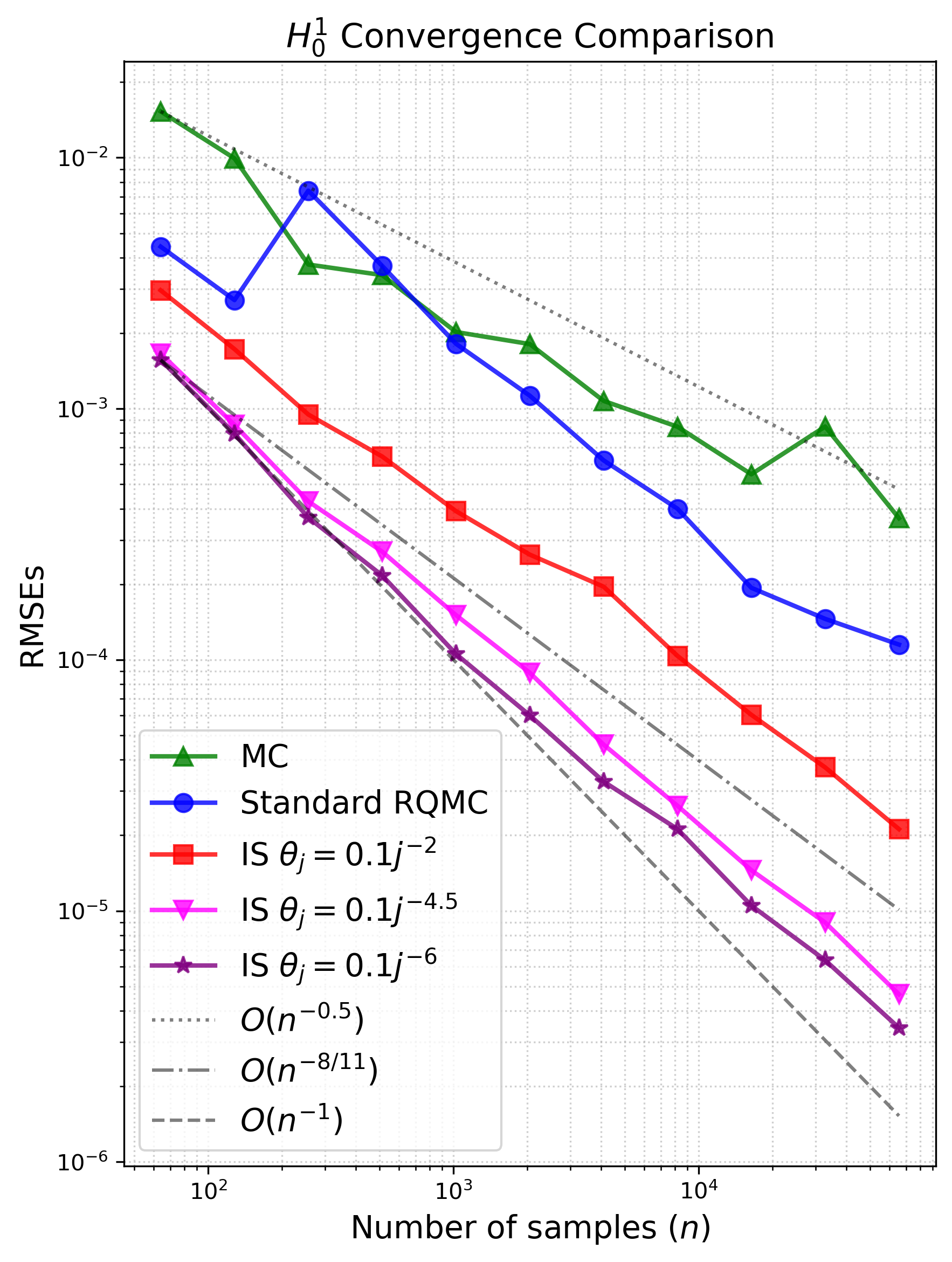} 
        \caption{$H_0^1(D)$ Error}
        \label{fig:conv_h1}
    \end{subfigure}
    \hfill 
    \begin{subfigure}[b]{0.48\textwidth}
        \centering
        \includegraphics[width=\textwidth]{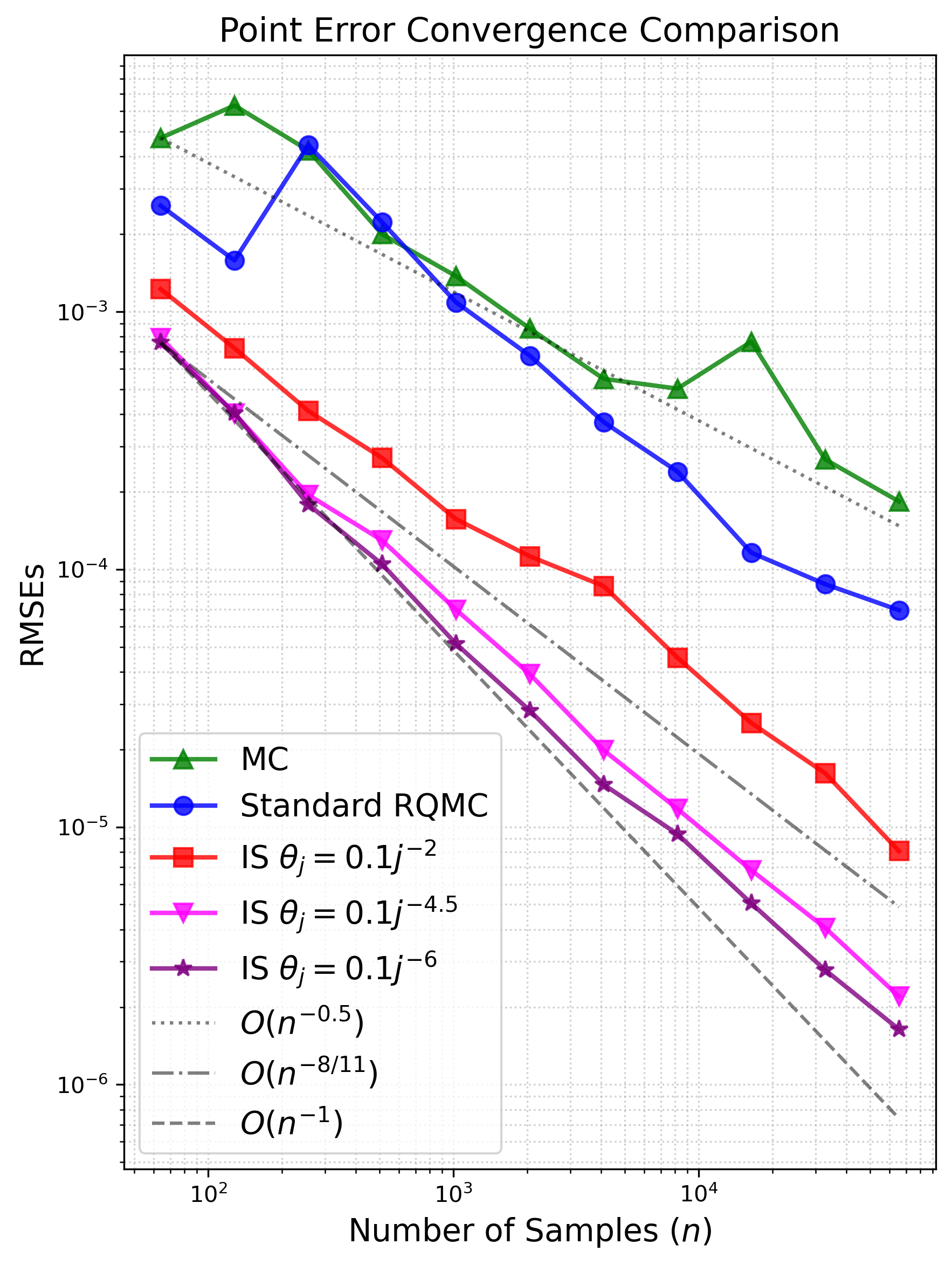} 
        \caption{Pointwise Error at $\boldsymbol{x}_c$}
        \label{fig:conv_point}
    \end{subfigure}

    \caption{Convergence results of the RMSEs for  $s=128$, $q\in (1,3/2)$, $\ \beta=\tau=1,\ \zeta = 2/3,\ \rho^*=4$. The plots compare the proposed BDIS method ($\theta_{j}=0.1j^{-2}$, $0.1j^{-4.5}$ or $0.1j^{-6}$) against standard RQMC and MC methods. The dashed reference lines indicate $O(n^{-0.5})$, $O(n^{-8/11})$ and $O(n^{-1})$ convergence, respectively.}
    \label{fig:convergence_results}
\end{figure}

Figure~\ref{fig:convergence_results2} reports the RMSEs for the case $\rho^*=2, \zeta=1/2$, and $\beta=1$, for which the integrand belongs to $W^{1,q}_{\mix}(\real^s,\varphi)$ with $q\in (1,2)$. Corollary~\ref{cor:pod} calls for  the choice $\theta_j=0.1j^{-2}$, while omitting the factorial term motivates the alternative $\theta_j=0.1j^{-4}$. Compared to the previous setting, the value $q$ increases to nearly $2$, and we observe a corresponding improvement in the RMSE rate of standard RQMC. Since $q$ cannot be taken arbitrarily large, there remains room for improvement over standard RQMC via BDIS.  Our BDIS method with $\theta_j=0.1j^{-4}$ and $\theta_j=0.1j^{-6}$ clearly outperforms standard RQMC, while the choice $\theta_j=0.1j^{-2}$ performs slightly worse than standard RQMC. These findings again suggest that omitting the factorial term leads to a better estimate of the optimal choice for $\theta_j$.

\begin{figure}[ht]
    \centering
    \begin{subfigure}[b]{0.48\textwidth}
        \centering
        \includegraphics[width=\textwidth]{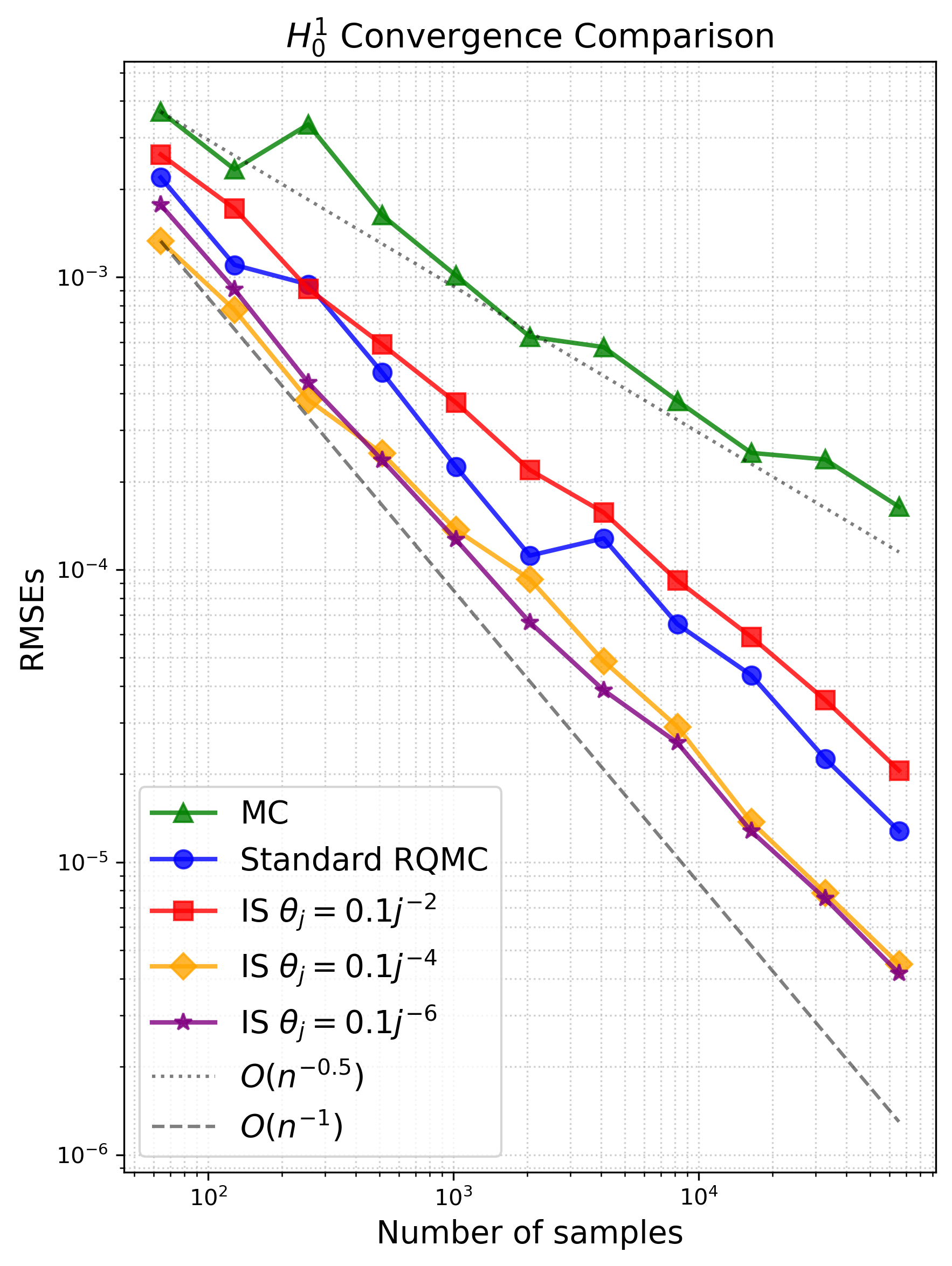} 
        \caption{$H_0^1(D)$ Error}
    \end{subfigure}
    \hfill 
    \begin{subfigure}[b]{0.48\textwidth}
        \centering
        \includegraphics[width=\textwidth]{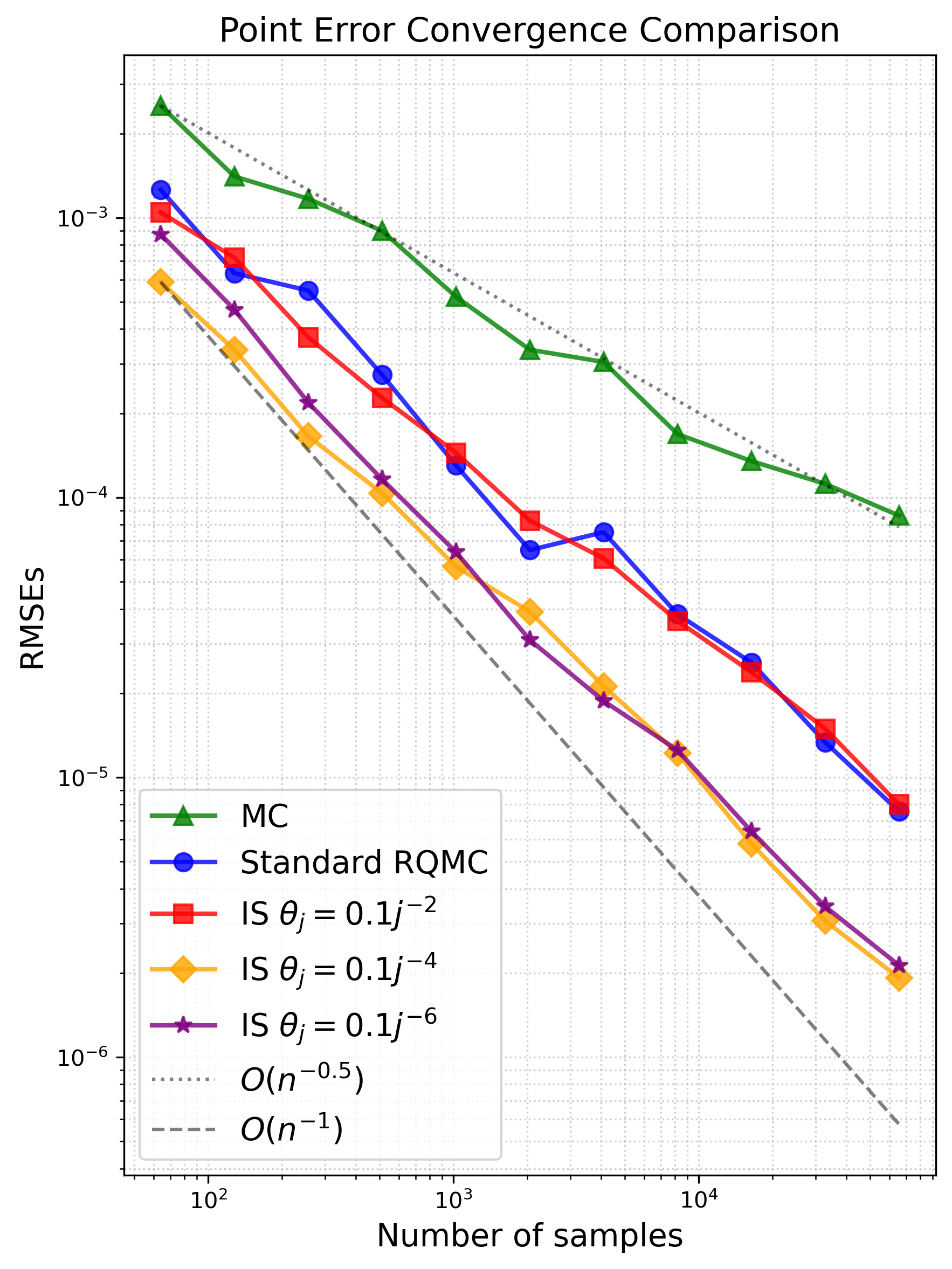} 
        \caption{Pointwise Error at $\boldsymbol{x}_c$}
    \end{subfigure}

    \caption{Convergence results of the RMSEs for  $s=128,\ q\in (1,2),\ \beta=\tau=1,\ \zeta = 1/2,\ \rho^*=2$. The plots compare the proposed BDIS method ($\theta_{j}=0.1j^{-2}$, $0.1j^{-4}$ or $0.1j^{-6}$) against standard RQMC and MC methods. The dashed reference lines indicate $O(n^{-0.5})$ and $O(n^{-1})$ convergence, respectively.}
    \label{fig:convergence_results2}
\end{figure}

Finally, to examine the regime of $\beta>\tau=1$, we consider lognormal random fields, for which $\beta=2$ and $q$ can be arbitrarily large.  We set $\zeta=2/3$ and $\rho^*=5/2$, yielding an easy test case for QMC integration. Corollary~\ref{cor:pod} suggests using  BDIS with $\theta_j=0$, which reduces to the standard RQMC method since $\rho=3/2$. This yields a dimension-independent RMSE rate of almost $O(n^{-1})$, as confirmed in Figure~\ref{fig:convergence_results3}. For comparison, we also test BDIS with three decay rates for $\theta_j$:  $\theta_j = 0.1 j^{-2},\ \theta_j =0.1j^{-4}$, and $\theta_j = 0.1j^{-6}$. Figure~\ref{fig:convergence_results3} shows that BDIS achieves a faster error rate when $\theta_j$ decays more rapidly. These results confirm the optimality of the standard RQMC in this setting.

% When $\rho^*=2.1$, Corollary~\ref{cor:pod} suggests to use $\theta_j = \theta_0 j^{-2.75}$ since $\rho=1.1\in (1,3/2)$, predicting a dimension-independent RMSE rate of almost $O(n^{-0.6})$. When $\rho^*=2.5$, 

\begin{figure}[ht]
    \centering
    \begin{subfigure}[b]{0.48\textwidth}
        \centering
        \includegraphics[width=\textwidth]{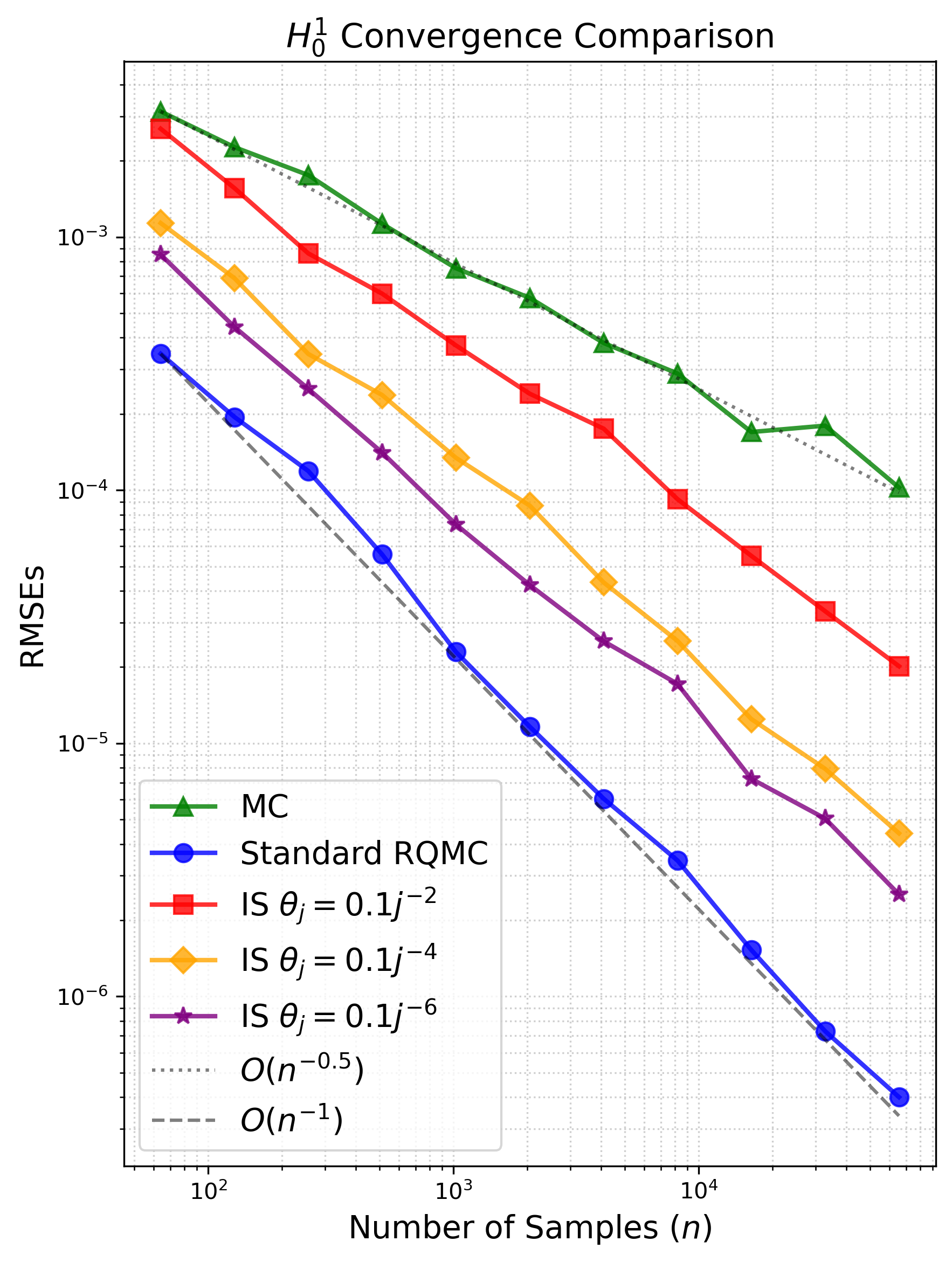} 
        \caption{$H_0^1(D)$ Error}
    \end{subfigure}
    \hfill 
    \begin{subfigure}[b]{0.48\textwidth}
        \centering
        \includegraphics[width=\textwidth]{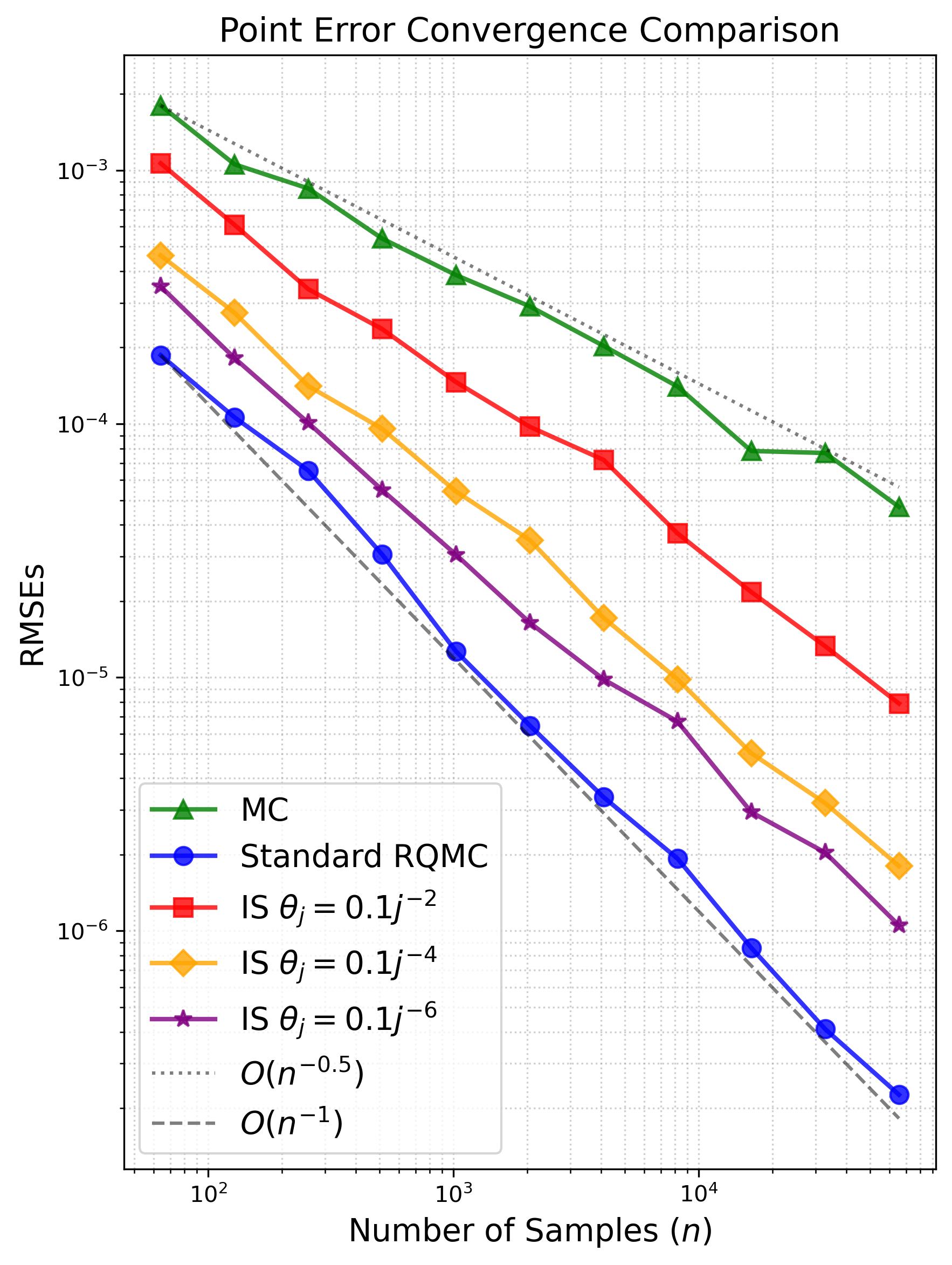} 
        \caption{Pointwise Error at $\boldsymbol{x}_c$}
    \end{subfigure}

    \caption{Convergence results of the RMSEs for  $s=128,\ q>1,\ \beta=2,\ \tau=1,\ \zeta = 2/3,\ \rho^*=5/2$. The plots compare the proposed BDIS method ($\theta_{j}=0.1j^{-2}$, $0.1j^{-4}$ or $0.1j^{-6}$) against standard RQMC and MC methods. The dashed reference lines indicate $O(n^{-0.5})$ and $O(n^{-1})$ convergence, respectively.}
    \label{fig:convergence_results3}
\end{figure}

\section{Concluding remarks}\label{sec:conclusion}

As a follow-up study of BDIS, this paper derives an improved error bound for scrambled digital nets and establishes a dimension-independent error rate for UQ problems in elliptic PDEs with random diffusion coefficients. The analysis is further extended to weighted function spaces, thereby bridging the proposed method with existing literature on constructive QMC approaches.

A natural extension of this work is to investigate whether BDIS can be combined with constructive methods such as lattice rules. Theorem~\ref{thm:Wqcase} provides an estimate of the $W^{1,1}_{\mix}(\itr^s)$-norm of the modified integrand $f^w$ in terms of the $W^{1,q}_{\mix}(\real^s,\varphi)$-norm of the original integrand $f$, which itself is bounded as in Theorem~\ref{thm:PODweight}. These estimates could be used to design suitable weights for constructing lattice rules. An interesting research direction is to explore whether jointly optimizing the transport maps $T_j$ and the lattice rules can lead to improved convergence and tractability.

Our analysis builds on first-order mixed derivatives of the integrand. A natural question is whether higher-order QMC rules \cite{dick:2008} can be combined with BDIS by exploiting higher-order derivatives. A new transport map is required to eliminate the discontinuities in the second derivative of $T_j$ defined by \eqref{eqn:Tj} and \eqref{eqn:wtheta}. We anticipate that a suitable construction would achieve convergence rates matching those of Möbius-transformed QMC rules \cite{kazashi2025optimalityquasimontecarlomethods,suzuki2025mobius}. We leave this question for future research.

\appendix
\section{Proof of Lemma~\ref{lem:Sobolevimbedding}}\label{appA}
\begin{proof}
    We first establish the result for the univariate case ($s=1$). For any $f \in C^1(\itr)$ and any $x, y \in \itr$, we have
    \begin{equation*}
        f(x) = f(y) + \int_y^x f'(t) \rd t.
    \end{equation*}
    Integrating both sides with respect to $y$ over $\itr$ gives
    \begin{equation*}
        f(x) = \int_0^1 f(y) \rd y + \int_0^1 \int_y^x f'(t) \rd t \rd y.
    \end{equation*}
    Taking the absolute value yields
    \begin{align*}
        |f(x)| &\leq \int_0^1 |f(y)| \rd y + \int_0^1 \left| \int_y^x f'(t) \rd t \right| \rd y \\
        &\leq \Vert f \Vert_{L^1(\itr)} + \int_0^1 \int_0^1 |f'(t)| \rd t \rd y \\
        &= \Vert f \Vert_{L^1(\itr)} + \Vert f' \Vert_{L^1(\itr)} = \Vert f \Vert_{W^{1,1}_{\mix}(\itr)}.
    \end{align*}
    Taking the supremum over $x \in \itr$ gives the desired inequality $\Vert f \Vert_{L^\infty(\itr)} \leq \Vert f \Vert_{W^{1,1}_{\mix}(\itr)}$.
    
    For the multivariate case ($s > 1$), we prove it by induction. Assume the inequality holds for all dimensions less than $s$. Let $\bsu = (\bsu', u_s) \in \itr^s$, where $\bsu' \in \itr^{s-1}$ and $u_s \in [0,1]$. Applying the univariate inequality with respect to the variable $u_s$ while holding $\bsu'$ fixed yields
    \begin{equation*}
        |f(\bsu', u_s)| \leq \int_0^1 |f(\bsu', t)| \rd t + \int_0^1 |\partial_s f(\bsu', t)| \rd t.
    \end{equation*}
    Taking the supremum over $\bsu = (\bsu', u_s)\in \itr^{s}$, we have
    \begin{equation} \label{eq:inductive_split}
        \sup_{\bsu \in \itr^s} |f(\bsu)| \leq \sup_{\bsu' \in \itr^{s-1}} \int_0^1 |f(\bsu', t)| \rd t + \sup_{\bsu' \in \itr^{s-1}} \int_0^1 |\partial_s f(\bsu', t)| \rd t.
    \end{equation}
    By the inductive hypothesis,
    \begin{equation*}
        \sup_{\bsu' \in \itr^{s-1}} |f(\bsu',t)| \leq \Vert f(\cdot,t) \Vert_{W^{1,1}_{\mix}(\itr^{s-1})}, \quad \sup_{\bsu' \in \itr^{s-1}} |\partial_s f(\bsu',t)| \leq \Vert \partial_s f(\cdot, t) \Vert_{W^{1,1}_{\mix}(\itr^{s-1})}.
    \end{equation*}
   We split the subsets $v\subseteq{1{:}s}$ into two groups: those not containing $s$, corresponding to $f(\cdot,t)$, and those containing $s$, corresponding to $\partial_s f(\cdot,t)$. Then
    \begin{align*}
        \int_0^1\Vert f(\cdot,t) \Vert_{W^{1,1}_{\mix}(\itr^{s-1})}\rd t &= \int_0^1\sum_{v \subseteq \{1{:}s-1\}} \int_{\itr^{s-1}} |\partial^v_{\bsu'} f(\bsu',t)| \rd \bsu' \rd t \\
        &= \sum_{v \subseteq \{1{:}s-1\}} \int_{\itr^{s}} |\partial^v f(\bsu)| \rd \bsu,
    \end{align*}
    and similarly,
    \begin{align*}
        \int_0^1\Vert \partial_s f(\cdot,t) \Vert_{W^{1,1}_{\mix}(\itr^{s-1})} \rd t &= \int_0^1\sum_{v \subseteq \{1{:}s-1\}} \int_{\itr^{s-1}} |\partial^v_{\bsu'} \partial_s f(\bsu',t)| \rd \bsu' \rd t \\
        &= \sum_{v \subseteq \{1{:}s-1\}} \int_{\itr^{s}}  |\partial^v (\partial_s f)(\bsu)|  \rd \bsu.
    \end{align*}
    Substituting these expressions into \eqref{eq:inductive_split} shows that the sum of the two terms covers all subsets $v\subseteq{1{:}s}$:
    \begin{align*}
        \Vert f \Vert_{L^\infty(\itr^s)} \leq & \sum_{\substack{v \subseteq \{1{:}s\} \\ s \notin v}} \Vert \partial^v f \Vert_{L^1(\itr^s)} + \sum_{\substack{v \subseteq \{1{:}s\} \\ s \in v}} \Vert \partial^v f \Vert_{L^1(\itr^s)} = \|f\|_{W^{1,1}_{\mathrm{mix}}(\itr^s)}.      
    \end{align*}
\end{proof}

\section{Proof of Lemma~\ref{lem:Poincareinequality}}\label{appB}
\begin{proof}
    First note that for $f \in W^{1,q}(\real,\varphi)$, 
    $$\| f - \mu(f)\|_{L^q(\real,\varphi)} \le \| f - f(0)\|_{L^q(\real,\varphi)} + \|  \mu(f) - f(0)\|_{L^q(\real,\varphi)} \le 2\| f - f(0)\|_{L^q(\real,\varphi)},$$
    where in the second inequality, we used the fact that
    \begin{align*} \|  \mu(f) - f(0)\|_{L^q(\real,\varphi)}  = & \left| \int_{\mathbb{R}} (f(x) - f(0)) \varphi(x)\rd x\right|\\
    \le & \left( \int_{\mathbb{R}}|f(x) - f(0)|^{q} \varphi(x)\rd x\right)^{1/q} 
    = \| f - f(0)\|_{L^q(\real,\varphi)}.
    \end{align*}
    Consequently, to prove the Poincaré  inequality~\eqref{eqn:Poincare}, it suffices to prove the following weighted Hardy inequality:
    \begin{equation*}
        \int_\real \left| \int_0^x f'(t) \rd t \right|^q \varphi(x) \rd x \leq C_q \int_\real |f'(x)|^q \varphi(x) \rd x,
    \end{equation*}
    where $C_q>0$ is a constant independent of $f(x)$.

    Denote $g = f'$. Due to the symmetry of $\varphi$,  it suffices to prove
    \begin{equation*}
        \int_{-\infty}^0 \left| \int_x^0 g(t) \rd t \right|^q \varphi(x) \rd x \leq C \int_{-\infty}^0 |g(x)|^q \varphi(x) \rd x.
    \end{equation*}
    According to \cite{muckenhoupt1972hardy}, this inequality holds if and only if
    \begin{equation} \label{eq:Muckenhoupt_cond}
        B^- := \sup_{x < 0} \left( \int_{-\infty}^x \varphi(t) \rd t \right)^{\frac{1}{q}} \left( \int_x^0 \varphi(t)^{-\frac{1}{q-1}} \rd t \right)^{\frac{q-1}{q}} < \infty.
    \end{equation} 
Because the Mills's ratio $m(x)$ equals $\Phi(-x)/\varphi(-x)$ for $x \geq 0$, our assumption implies there exists a constant $C>0$ satisfying $\varphi(t)^{-1} \leq C \Phi(t)^{-1}$ for all $t \leq 0$. Consequently,
    \begin{align*}
       \int_x^0 \varphi(t)^{-\frac{1}{q-1}} \rd t= \int_x^0 \varphi(t)^{-\frac{q}{q-1}} \varphi(t) \rd t &\le C^{\frac{q}{q-1}}\int_x^0 \Phi(t)^{-\frac{q}{q-1}} \varphi(t) \rd t \\
        &= C^{\frac{q}{q-1}}\int_x^0 \Phi(t)^{-\frac{q}{q-1}}  \rd \Phi(t)\\
        & = C^{\frac{q}{q-1}}(q-1) \left[ \Phi(x)^{-\frac{1}{q-1}} - \Phi(0)^{-\frac{1}{q-1}} \right].
    \end{align*}
    Substituting this bound into \eqref{eq:Muckenhoupt_cond}, we obtain
    \begin{align*}
        B^- 
        &\leq \sup_{x < 0} \Phi(x)^{\frac{1}{q}} \cdot \left( C^{\frac{q}{q-1}}(q-1) \Phi(x)^{-\frac{1}{q-1}} \right)^{\frac{q-1}{q}}= C(q-1)^{\frac{q-1}{q}} < \infty.
    \end{align*}
    This proves the desired result.
\end{proof}

\section{Proof of Lemma~\ref{lem:interchange}}\label{appC}
We prove the statement when $v$ is a singleton. The full statement follows by induction. Under our assumptions, the measure $\varphi(x)\rd x$ is equivalent to the Lebesgue measure over $\real$. We first consider the case $v\subsetneq w^c$ and write $f(\bsx)$ as $ f(x_{v},\bsx_{w},\bsx_{w^c\setminus v})$. For $h\neq 0$ and almost every $\bsx_{w^c\setminus v}\in \real^{s-|w|-1}$, 
    \begin{align*}
       &\frac{1}{h}\int_{\real^{|w|}} \left(f(x_{v}+h,\bsx_{w},\bsx_{w^c\setminus v})-f(x_v,\bsx_{w},\bsx_{w^c\setminus v})\right) \prod_{j\in w}\varphi(x_j)\rd \bsx_{w}\\
       =&\frac{1}{h} \int_{\real^{|w|}} \left(\int_{x_v}^{x_v+h} \partial_{v} f(t,\bsx_{w},\bsx_{w^c\setminus v})\rd t \right)\prod_{j\in w}\varphi(x_j)\rd \bsx_{w}
       =\frac{1}{h}\int_{x_v}^{x_v+h}g(t,\bsx_{w^c\setminus v}) \rd t,
    \end{align*}
   where  
    $$g(t,\bsx_{w^c\setminus v})=\int_{\real^{|w|}}  \partial_{v} f(t,\bsx_{w},\bsx_{w^c\setminus v})\prod_{j\in w}\varphi(x_j)\rd \bsx_{w}.$$
    Here, we can apply Fubini's theorem to change the order of integration because $\partial_{v} f(t,\bsx_{w},\bsx_{w^c\setminus v})\in L^{1}(\real^s,\varphi)$ and
    \begin{align*}
     &\int_{\real^{|w|}} \int_{x_v}^{x_v+h} |\partial_{v} f(t,\bsx_{w},\bsx_{w^c\setminus v})| \prod_{j\in w}\varphi(x_j)\rd t\rd \bsx_{w}\\
     \leq &
     \frac{1}{\min_{t\in [x_v,x_v+h]}\varphi(t)}\int_{\real^{|w|}} \int_{x_v}^{x_v+h} |\partial_{v} f(t,\bsx_{w},\bsx_{w^c\setminus v})|\varphi(t)\prod_{j\in w}\varphi(x_j)\rd t\rd \bsx_{w}<\infty
    \end{align*}
    for every $\bsx_{w^c\setminus v}\in \real^{s-|w|-1}$ except for a measure zero set $A\subseteq \real^{s-|w|-1}$.
The above argument also shows $g(t,\bsx_{w^c\setminus v})$ is locally integrable in $t$ for every $\bsx_{w^c\setminus v}\in \real^{s-|w|-1}\setminus A$. Hence, we can apply Lebesgue differentiation theorem and conclude when $\bsx_{w^c\setminus v}\in \real^{s-|w|-1}\setminus A$,
\begin{equation}\label{eqn:Lebgdiff}
  \lim_{h\to 0}\frac{1}{h}\int_{x_v}^{x_v+h}g(t,\bsx_{w^c\setminus v}) \rd t=g(x_v,\bsx_{w^c\setminus v})  
\end{equation}
    for every $x_v\in \real$
    except for a measure zero set $B_{\bsx_{w^c\setminus v}}\subseteq \real$ that may depend on $\bsx_{w^c\setminus v}$. As shown in \cite[Chapter 4.1]{nikol2012approximation}, the union $\mathcal{B}=\bigcup_{\bsx_{w^c\setminus v}\in \real^{s-|w|-1}\setminus A} B_{\bsx_{w^c\setminus v}}\times \{\bsx_{w^c\setminus v}\}$ is measurable and has measure zero  with respect to the Lebesgue measure over $\real^{s-|w|}$. It then follows that
    \begin{equation*}
  \partial_v\int_{\real^{|w|}} f(x_v,\bsx_{w},\bsx_{w^c\setminus v}) \prod_{j\in w}\varphi(x_j)\rd \bsx_{w}=\int_{\real^{|w|}}  \partial_{v} f(t,\bsx_{w},\bsx_{w^c\setminus v})\prod_{j\in w}\varphi(x_j)\rd \bsx_{w}      
    \end{equation*}
 for every $\bsx_{w^c}=(x_v,\bsx_{w^c\setminus v})\in \real^{s-|w|}$ except for the measure zero set $(\real\times A) \bigcup \mathcal{B}$. 
 
 The case $v=w^c$ follows directly from \eqref{eqn:Lebgdiff}, which holds for almost every $x_v\in \real$.
%%\section*{Acknowledgments}

\end{document}